\documentclass[10pt]{amsart}
\usepackage{amsfonts}
\usepackage{enumerate}
\usepackage{latexsym}
\usepackage{graphicx}
\usepackage{amssymb,latexsym,amsmath,amsthm}

\begin{document}

\newcommand{\h}[2]{\ensuremath{H ^{#1} _{#2}}}
\newcommand{\Hdot}{\ensuremath{\dot{H}}}
\newcommand{\hdot}[2]{\ensuremath{\dot{H} ^{#1} _{#2}}}
\newcommand{\vdot}{\ensuremath{\dot{v}}}
\newcommand{\Udot}{\ensuremath{\dot{U}}}
\newcommand{\Uone}{\ensuremath{^{(1)}U}}
\newcommand{\Utwo}{\ensuremath{^{(2)}U}}
\newcommand{\Uthree}{\ensuremath{^{(3)}U}}
\newcommand{\Hone}{\ensuremath{^{(1)}H}}
\newcommand{\Htwo}{\ensuremath{^{(2)}H}}
\newcommand{\Hthree}{\ensuremath{^{(3)}H}}
\newcommand{\hone}[2][2]{\ensuremath{^{(1)}H^{#1} _{#2}}}
\newcommand{\htwo}[2][2]{\ensuremath{^{(2)}H^{#1} _{#2}}}
\newcommand{\hthree}[2][2]{\ensuremath{^{(3)}H^{#1} _{#2}}}
\newcommand{\vone}{\ensuremath{^{(1)}v}}
\newcommand{\vtwo}{\ensuremath{^{(2)}v}}
\newcommand{\vthree}{\ensuremath{^{(3)}v}}
\newcommand{\vbar}{\ensuremath{\bar{v}}}
\newcommand{\deet}{\ensuremath{\partial _t}}
\newcommand{\deei}{\ensuremath{\partial _i}}
\newcommand{\deel}{\ensuremath{\partial _l}}
\newcommand{\deek}{\ensuremath{\partial _k}}
\newcommand{\deem}{\ensuremath{\partial _m}}
\newcommand{\deej}{\ensuremath{\partial _j}}
\newcommand{\deep}{\ensuremath{\partial _p}}
\newcommand{\deer}{\ensuremath{\partial _r}}
\newcommand{\deen}{\ensuremath{\partial _n}}
\newcommand{\deerho}{\ensuremath{\partial _\rho}}
\newcommand{\deetau}{\ensuremath{\partial _\tau}}
\newcommand{\omegatwiddle}{\ensuremath{\tilde{\Omega}}}
\newcommand{\Stwiddle}{\ensuremath{\tilde{S}}}
\newcommand{\norm}[1]{\ensuremath{\| #1 \|}}
\newcommand{\ltwonorm}[1]{\ensuremath{\| #1 \| _{L^2(\mathbf{R}^3)}}}
\newcommand{\linfnorm}[1]{\ensuremath{\| #1 \| _{L^\infty}}}
\newcommand{\etalinfnorm}[1]{\ensuremath{\| #1 \| _{L^\infty(\eta\geq 0)}}}
\newcommand{\visc}{\ensuremath{\nu}}
\newcommand{\innerprod}[2]{\ensuremath{\langle #1 ,#2 \rangle}}
\newcommand{\rthreeip}[2]{\ensuremath{\langle #1 ,#2 \rangle _{\mathbf{R}^3}}}
\newcommand{\ltwoip}[2]{\ensuremath{\langle #1 ,#2 \rangle _{L^2(\mathbf{R}^3)}}}
\newcommand{\E}[2]{\ensuremath{E_{#1,#2}}}
\newcommand{\Etwiddle}[2]{\ensuremath{\tilde{E}_{#1,#2}}}
\newcommand{\stwiddleoops}[2]{\ensuremath{\tilde{S}^{#1} \Upsilon ^{#2}}}
\newcommand{\soops}[2]{\ensuremath{S^{#1} \Upsilon ^{#2}}}
\newcommand{\ahat}[2]{\ensuremath{\hat{A} ^{#1} _{#2}}}
\newcommand{\bhat}[2]{\ensuremath{\hat{B} ^{#1} _{#2}}}
\newcommand{\bee}[2]{\ensuremath{B ^{#1} _{#2}}}
\newcommand{\grad}{\ensuremath{\nabla}}
\newcommand{\ctwot}{\ensuremath{\langle t \rangle}}
\newcommand{\laplacian}{\ensuremath{\triangle}}
\newcommand{\tr}{\textrm{tr }}
\newcommand{\vtwiddle}{\ensuremath{\tilde{v}}}
\newcommand{\Htwiddle}{\ensuremath{\tilde{H}}}
\newcommand{\ftwiddle}{\ensuremath{\tilde{f}}}
\newcommand{\gtwiddle}{\ensuremath{\tilde{g}}}
\newcommand{\Utwiddle}{\ensuremath{\tilde{U}}}
\newcommand{\yconst}{\ensuremath{\varrho}}
\newcommand{\Pee}[1]{\ensuremath{\textsf{P}_{#1}}}
\newcommand{\eigensp}[1]{\aleph _{#1}(\omega)}
\newcommand{\sproj}[1]{\ensuremath{\mathbb{P}_{#1}}}
\newcommand{\ev}[1]{\lambda _{#1}}
\newcommand{\Chi}[2]{\ensuremath{\mathcal{X}_{#1,#2}(\dot{U})}}
\newcommand{\projind}{\ensuremath{\iota}}
\newcommand{\rthree}{\ensuremath{{\mathbf{R}^3}}}
\newcommand{\thesistitle}{Global Existence with Small Initial Data for Three-Dimensional Incompressible Isotropic Viscoelastic Materials}
\newcommand{\name}{Paul Kessenich}

\newtheorem{theorem}{Theorem}[section]
\newtheorem{lemma}{Lemma}[section]
\newtheorem{corollary}{Corollary}[section]

\begin{center}
{\Large \bf \thesistitle}

\normalsize

\vspace{.2in}

{\large PAUL KESSENICH}

\textit{University of Michigan}

\vspace{.2in}

{\bf Abstract}

\begin{quote}

Global existence for a system of nonlinear partial differential equations (PDE) modeling an isotropic incompressible viscoelastic material is proved.  The structure of the PDE is derived through constitutive assumptions on the material.  Restriction on the size of the initial displacement and velocity for the model is specified independent of the size of the viscosity of the material.  The proof of global existence combines use of vector fields, local energy decay estimates, generalized Sobolev inequalities, and hyperbolic energy estimates. 
\end{quote}
\end{center}

\newpage

\section{Introduction}

This paper is concerned with the Cauchy problem for three dimensional isotropic viscoelasticity.  The equations for the motion of the viscoelastic material are derived from constitutive assumptions.  The resulting system of equations is a hybrid between the equations for a viscous Newtonian fluid  and a Cauchy elastic material.  These equations can be viewed as a quasilinear symmetric hyperbolic system with a parabolic perturbation of a certain form arising from the viscosity.  Studying the system in this way, this article presents a result concerning the global stability of the motion of this material.

The main theorem states that the system has global solutions in time given sufficiently small initial displacement and velocity, and furthermore, the size of these initial data are independent of the magnitude of the viscosity.  The methods used for this result include generalized hyperbolic energy estimates, generalized Sobolev inequalities, and local energy decay.  The parabolic term is avoided for the most part and does not provide additional decay during the course of the proof.

Although many treatments of elasticity use Lagrangian coordinates, the equations will be studied as a first order system in Eulerian coordinates with constraints as in \cite{ST2}.  Using these coordinates, the form of the constraints is particularly useful and the portion of the linear operator due to the pressure waves is easily seen to be higher order.  Written in Eulerian variables, the resulting system is a mixed type PDE, an interpolation between the Navier-Stokes equations and incompressible isotropic elastodynamics.

The proofs in this article utilize previous techniques and insight from the study of 3D quasilinear wave equations and elasticity.  Global existence cannot be expected from a general symmetric hyperbolic system.  For certain types of ``genuine nonlinearity" John proved blow up for the quasilinear wave equation in which the second derivatives of the solution blow up in $L^2$ even for arbitrarily small initial data \cite{John1}.  In turn, he showed a blow up result for elasticity using similar strategies on equations containing these types of ``genuine" nonlinear terms \cite{John3}.  As a complement to John's ``genuine nonlinearity", Christodoulou \cite{Christo1} and Klainerman \cite{Kl4} concurrently identified a ``null condition" for the quadratic terms of the wave equation.  When all the quadratic terms satisfy this null condition, small data global existence can be proved.  Klainerman's techniques involve proving strong dispersive estimates and using of Lorentzian and Gallilean invariant vector fields in the energy without directly estimating the fundamental solution.

Analogous global stability results for nonlinear elasticity could not immediately be realized because the equations contain multiple wave speeds in the linear operator causes a lack of Lorentz invariance.  For elastic materials which are Galilean invariant and isotropic, works by John \cite{John4} and later Klainerman and Sideris \cite{KS1} showed ``almost global" existence for small initial using a smaller set of vector fields.  In \cite{Sid1} Sideris went further, introducing a null condition (similar to that of Klainerman and Christodoulou) for both the shear and pressure waves present in compressible elasticity. He showed that the shear waves inherently satisfy the null condition while the pressure waves generally may not.  Using these null conditions and the previous vector field methods he proved global existence for compressible elastic materials close to equilibrium without estimating the fundamental solution.  Independently, Agemi achieved the same result using null conditions and estimation of the fundamental solution \cite{Agemi1}.  Blow up for elasticity with large initial data was shown by Tahvilar-Zadeh in \cite{TZ1}.

Small data global existence for incompressible elasticity was proven by Sideris and Thomases using the incompressible limit in \cite{ST1} and directly in \cite{ST2}.  Both results use vector field methods and dispersive estimates to prove global existence.  In the case of isotropic elasticity only the inherent null condition for shear waves is needed because the pressure waves do not appear, so no additional assumptions were made on the material in either result.  In \cite{ST3}, Sideris and Thomases obtained dispersive estimates in the framework of a more general result on local energy decay for symmetric hyperbolic systems.  Although systems with certain dissipative terms are handled in these estimates, the techniques cannot be directly applied to the viscoelastic case because the viscosity does not appear in all components of the PDE.

Previous results on the Cauchy problem for incompressible three dimensional viscoelasticity by Lei, Liu, and Zhou \cite{LLZ2}, Lin, Liu, and Zhang \cite{LLZ1}, and also Chen and Zhang \cite{ChenZhang} use the parabolic structure of the equations to achieve global existence.  In each article, the authors study the Oldroyd B model for viscoelasticity, equations with a structure nearly identical to the equations studied herein.  Their global existence proofs involve writing the PDE in terms of a ``special quantity" which is a specific linear combination of the deformation gradient and the Eulerian gradient of the velocity.  This change of variables reveals additional dissipative structure which is then used to prove small data global existence.  For these results, the magnitude of the initial data must be small relative to the viscosity.  Consequently, convergence to the equations for elastodynamics cannot be studied in this way because the initial data vanishes with the viscosity.

To study global solutions for viscoelasticity from the perspective of 3D quasilinear hyperbolic PDE, the paper is organized as follows.  The PDE are derived from constitutive assumptions on the material.  After introducing some necessary notation and stating the main theorem, local energy estimates similar to those of \cite{ST3} are computed.  Projections are then defined on the eigenspaces of the symbol of the hyperbolic operator, and a null condition for shear waves is stated in terms of the projections.  As usual the pressure is shown to be bounded by nonlinear terms, and weighted Sobolev inequalities are stated.  A new addition to these estimates is combined with a Hardy-type inequality to give the key inequality (\ref{sob6}).  Following bootstrapping of the local energy estimates and computation of the energy identity, all of the pieces are pulled together to prove global existence.

\section{Derivation of Equations}

\setcounter{equation}{0}
\renewcommand{\theequation}{2.\arabic{equation}}

First, assume our substance is a continuous distribution of matter at rest and that it fills three dimensional space so that each point $X=(X_1,X_2,X_3) \in \rthree$ corresponds to a point in the substance.  These points $X$ are called material or Lagrangian coordinates.  If we deform our material in a differentiable manner to a different configuration at a later time $t$, we call these new coordinates spatial or Eulerian coordinates and label them $x(t,X)$.  The map $x:\mathbf{R} \times \rthree \rightarrow \rthree$ is an orientation preserving diffeomorphism.  The inverse map $X(t,x)$ returns the material point that has deformed to spatial coordinate $x$ at time $t$.  The deformation gradient is the matrix
$$F = \frac{\partial x}{\partial X}$$
with inverse
$$H = \frac{\partial X}{\partial x}.$$
Both $H$ and $F$ must have positive determinant because of our assumption that $x$ preserves orientation.  We denote $$D=\frac{\partial}{\partial X}=(D_1,D_2,D_3)\quad \mbox{and} \quad \grad = \frac{\partial }{\partial x}=(\partial _1,\partial _2,\partial _3)$$
as well as $D_t$ for Lagrangian time derivatives and $\partial _t$ for Eulerian time derivatives.  Using the chain rule, we have the relations
$$D_j = F ^p _j \deep \quad \mbox{and} \quad D_t = \deet + v\cdot \grad$$
where $v=D_t x$.  Here and in what follows, we use the Einstein summation notation where summing over repeated indices is understood.

Constitutive assumptions must be made to further specify material properties.  For an elastic solid the equations of motion are typically derived from a variational problem, however the dissipative motion of a viscoelastic material prevents us from obtaining our equations in this way.  Therefore, we begin with the force balance laws expressing the conservation of mass and momentum:
\begin{align}
&D_t \rho +\grad \cdot v \rho=0 \label{conslawlag}\\
&\rho D_t v - \grad \cdot T = \rho f. \notag
\end{align}
where $\rho(t,x)$ is the density, $f(t,x)$ is the external force (which we assume to be $0$), and $T(t,x)$ is the Cauchy stress.  If we write these equations with Eulerian time derivatives , we have

\begin{align}
&\deet \rho + v\cdot \grad \rho + \grad \cdot v \rho =0 \label{conslaweul}\\
&\rho \deet v + v\cdot \grad v - \grad \cdot T = 0. \notag
\end{align}

We first assume the material is incompressible.  This ensures the density is constant and transforms the conservation of mass equation into the constraint
\begin{equation}
\grad \cdot v = 0. \label{divv}
\end{equation}
This assumption also means volumes are conserved under the motion so that
\begin{equation}
\det F = \det H = 1. \label{deth}
\end{equation}
For simplicity, we will assume the density is unity for the rest of the paper.

The Cauchy stress $T$ encodes the internal self-interacting forces of the material.  For a viscoelastic material $T$ depends on the pressure $p(t,x)$ (linearly), the deformation gradient $F$ and its Lagrangian time derivative $D_t F$.  By the chain rule, we can write the ODE
$$D_t F = \grad v F$$
so the Cauchy stress can be written as
$$T = -pI + \bar{T}(F, \grad v)$$
where $I$ is the identity matrix.  Assuming objectivity (Galiean invariance) of our material implies that $\bar T$ depends on $\grad v$ through the rate of strain tensor $D^i _j=\frac{1}{2}(\deej v^i + \deei v^j)$.  If we assume this dependence, which determines the internal frictional forces within the material, is linear the Cauchy stress takes the form
\begin{equation}
T = pI + \visc _0 D + \hat{T} (F). \label{teq}
\end{equation}

Here the tensor $\hat{T}$ contains the information about the elastic properties of the material.  If we assume the elastic forces come from an isotropic and objective strain energy function $W(F)$ as in \cite{Sid1,ST2}, then we have
$$W(FU)=W(F)=W(UF)$$
for all proper orthogonal matrices $U$.  The isotropy and objectivity assumptions imply that $W$ depends on $F$ through the principal invariants of the strain matrix $F^T F$.  The Piola-Kirchoff stress, $S$, is the Lagrangian version of the Cauchy stress and is defined as
$$S = \frac{\partial W}{\partial F}. $$
In this case, since our material is incompressible and $\rho =1$, $\hat{T}$ is related to $W$ by
\begin{equation}
F^j_p \frac{\partial W}{\partial F^i_p}(F) = F^j_p S^i_p(F) =  \hat{T}^i_j(F). \label{trelw}
\end{equation}
We will assume the material is stress free at the identity, i.e.
$$\hat{T}(I) = S(I) = 0.$$
This restriction rules out the Oldroyd B model.

Taking Eulerian divergence of (\ref{trelw}) and using $\deej F^j_p =0$ (a consequence of incompressibility), we have
\begin{equation} \label{divt}
[\grad \cdot \hat{T} (F)]^i = \deej (F^j_p \frac{\partial W}{\partial F^i_p}(F)) = \frac{\partial^2 W}{\partial F^i_p \partial F_k ^n } (F) F^j_p \deej F _k ^n \equiv A^{pk}_{in} (F) F^j_p \deej F _k ^n
\end{equation}
where $A=\frac{\partial S}{\partial F}$ is the elasticity tensor.  $A$ is symmetric, i.e. $A^{pk}_{in}=A^{kp}_{ni}$, because
\begin{equation}
\frac{\partial^2 W}{\partial F^i_p \partial F_k ^n } = \frac{\partial^2 W}{ \partial F_k ^n \partial F^i_p}.  \label{d2w}
\end{equation}
We remark that this form of the elasticity tensor is consistent with the notation in \cite{Sid1,ST2}.

We will not make any further assumptions about the nonlinear terms of these elasticity tensors, however, to be consistent with linear elasticity theory, we impose the Legendre-Hadamard ellipticity condition
\begin{equation} \label{LH}
A^{jk}_{in} (I) \xi_j \xi_k \omega^i \omega^n > 0
\end{equation}
for all $\xi,\omega \in S^2$, the unit sphere in $\rthree$.  This is a standard assumption on the form of the linear operator ruling out the incompressible Navier-Stokes equations.  Under the objectivity and isotropy assumptions this condition is equivalent to
$$A^{jk}_{in} (I) = (c_1^2 - 2c_2 ^2)\delta ^j_i \delta ^k_n + c_2 ^2 (\delta ^{jk} \delta _{in}+\delta ^j_n \delta ^k_i) $$
where $\delta$ denotes the unit tensor and the parameters $c_1$ and $c_2$ are positive constants representing the propagation speeds of the pressure and shear waves respectively.

Because of the convenience of the relation
\begin{equation}
\deel \h{i}{j} = \frac{\partial}{\partial x^l}   \frac{\partial X^i}{\partial x^j} = \frac{\partial}{\partial x^j}   \frac{\partial X^i}{\partial x^l} = \deej \h{i}{l} \label{swapind}
\end{equation}
we will write our PDE in terms of $H=F^{-1}$ defined above. Using the fact that $FH=I$ we have
$$\deep F_k^n = -F^n_r F^q_k \deep \h{r}{q}.$$
So we can write
\begin{align}
A^{jk}_{in}(F) F_j^p \deep F _k ^n &= A^{jk}_{ln}(F) \delta ^l _i F_j^p \deep F _k ^n \label{atoahat}
\\
&=A^{lk}_{jn}(F) F^l_m H^m_i F_j^p \deep F _k ^n \notag
\\
&=-A^{lk}_{jn}(F) F^l_m  F_j^p F^n_r F^q_k H^m_i \deep \h{r}{q}. \notag
\end{align}
If we add the null Lagrangian $\mathcal{L}^{pq}_{mr}=c_2^2(\delta^p_m \delta^q_r -\delta^p_r \delta^q_m)$, we can define
\begin{equation} \label{ahat}
\ahat{pq}{mr} (H) = A^{lk}_{jn}(H^{-1}) [H^{-1}]^l_m  [H^{-1}]_j^p [H^{-1}]^n_r [H^{-1}]^q_k + \mathcal{L}^{pq}_{mr}
\end{equation}
so that $\hat{A}(I)$ is positive definite.  This does not change our equations since
$$\mathcal{L}^{pq}_{mr} H^m_i \deep \h{r}{q} = 0$$
via the constraint (\ref{swapind}). $\hat{A}$ satisfies

$$\ahat{pq}{mr} (H) =\ahat{qp}{rm} (H)$$
$$\ahat{pq}{mr} (I) = (c_1^2-c_2^2)\delta ^p_m \delta ^q_r + c_2^2 \delta ^{pq} \delta _{mr} $$
and, also, for $|\Hdot|\equiv|H-I|$ sufficiently small,
\begin{equation} \label{propa}
\ahat{pq}{mr} (H) \hdot{m}{p} \hdot{r}{q} \geq c_2^2|\Hdot|^2,
\end{equation}
by continuity of $\hat{A}$.

We are now in a position to write our PDE in terms of $H$ and $v$.  First, by (\ref{teq}), (\ref{divt}) and (\ref{atoahat}) we have
\begin{align}
(\grad \cdot T)^i &= -\deej p\delta ^j_i +\frac{\visc _0}{2} \deej(\deej v^i +\deei v^j) + (\grad \cdot \hat T)^i \notag \\
 &= -\deei p + \visc \laplacian v^i - \ahat{lm}{pj} (H) H^p_i \deel \h{j}{m} \notag
\end{align}
where $\visc = \frac{\visc _0}{2}$ and we used the incompressibility constraint $\grad \cdot v = 0$.  Our conservation of momentum equation becomes
$$\deet v +v\cdot \grad v +\grad p + \hat{N} ^v (H) =\visc \laplacian v$$
where $\hat{N} ^v (H)^i = \ahat{lm}{pj} (H) \h{p}{i} \deel \h{j}{m}$.

Using the chain rule and the definitions of $H$ and $v$ we have the transport equation
$$\deet H+ v\cdot \grad H+H\grad v = 0.$$
Thus, the equations we have derived for a 3D isotropic incompressible viscoelastic material are
\begin{subequations}
\begin{align}
\deet H+ v\cdot \grad H+H\grad v &= 0 \label{heq1} \\
\deet v +v\cdot \grad v +\grad p + \hat{N} ^v (H) &=\visc \laplacian v. \label{heq2}
\end{align}
\end{subequations}
As we have already seen, the constraints we have accompanying these PDE are
\begin{subequations}
\begin{align}
\grad \cdot v &= 0 \label{hconst1} \\
\det H &=1 \label{hconst2} \\
\deej \h {i}{k} &= \deek \h{i}{j}
\label{hconst3}
\end{align}
\end{subequations}

At this point we remark that the Oldroyd B model for viscoelasticity corresponds to the case where above
$$\ahat{lm}{pj} (H) = (H^{-1})^n_j(H^{-1})^l_q(H^{-1})^m_q(H^{-1})^n_p.$$
Although this form of $\hat{A}$ cannot be derived via our method, the form of the equations is the same and the constraints are precisely the same as they arise solely from incompressibility.  The Oldroyd B system is typically written in terms of the deformation gradient $F$ as
\begin{subequations}
\begin{align}
\deet F+ v\cdot \grad F-F\grad v &= 0 \label{feq1} \\
\deet v +v\cdot \grad v +\grad p -\grad \cdot (FF^T) &=\visc \laplacian v. \label{feq2}
\end{align}
\end{subequations}
To verify equations (\ref{heq1})-(\ref{heq2}) are equivalent to (\ref{feq1})-(\ref{feq2}) in this case, one can check using $F=H^{-1}$ that
$$(H^{-1})^n_j(H^{-1})^l_q(H^{-1})^m_q(H^{-1})^n_p \h{p}{i} \deel \h{j}{m} = -[\grad \cdot (FF^T)] ^i.$$
We also note that the null condition which is discussed in section 6 is also satisfied in the Oldroyd B case.  Therefore, Theorem \ref{main} holds for the Oldroyd B model as well.

Turning again to the general case, we linearize the system (\ref{heq1})-(\ref{heq2}) using the notation $\Hdot =H-I$ and $v=\vdot$ obtaining the equations
\begin{subequations}
\begin{align}
\deet \Hdot+ \grad \vdot &= \tilde{N}^H (\Hdot, \vdot ) \label{hdoteq1} \\
\deet \vdot + \grad \cdot T\Hdot +\grad p -\visc \laplacian \vdot &= \tilde{N}^v (\Hdot, \vdot )
\label{hdoteq2}
\end{align}
\end{subequations}
where
\begin{align*}
(\tilde{N}^H)^i _j (\Hdot, \vdot ) &= -\vdot ^p \deep \hdot {i}{j} +\hdot {p}{j} \deep \vdot ^i
\\
(\tilde{N}^v)^i (\Hdot, \vdot ) &= -\vdot ^p \deep v ^i - \ahat{lm}{pj} (H) \hdot{p}{i} \deel \hdot{j}{m} - [\hat{A} (H)- \hat{A} (I)]^{lm}_{ij} \deel \hdot{j}{m}
\\
(\grad \cdot T\Hdot )^i &= \ahat {lm}{ij} (I) \deel \hdot{j}{m} = \deek [c_2^2 \hdot{i}{k}
+(c_1^2-c_2^2)(\tr \Hdot) \delta ^i_k].
\end{align*}
Our constraints then become
\begin{subequations}
\begin{align}
\grad \cdot \vdot &= 0 \label{hdotconst1} \\
\tr \Hdot &= \frac{1}{2} ((\tr \Hdot)^2-\tr(\Hdot ^2))+\det \Hdot \label{hdotconst2} \\
\deej \hdot {i}{k} &= \deek \hdot{i}{j} \label{hdotconst3}
\end{align}
\end{subequations}
Notice that the incompressibility of the system implies that $\tr \Hdot$ is higher order in \Hdot.  Because of this, we will move the $\tr \Hdot$ portion of the linearity onto the right side of (\ref{hdoteq2}).  For simplicity, we will assume herein that $c_2=1$.  We caution the reader that our dot notation $\Hdot, \vdot$ does not denote derivatives, consistent with the convention in \cite{ST3}.  After a bit of rearranging we have the equations in the form we will work with most frequently:
\begin{subequations}
\begin{align}
\deet \Hdot+  \grad \vdot &= N^H (\Hdot, \vdot) \label{hdoteqc2a}\\
\deet \vdot +  \grad \cdot \Hdot -\visc \laplacian \vdot &= N^v (\Hdot, \vdot)-(c_1^2-1)M^H(\Hdot) - \grad p
\label{hdoteqc2b}
\end{align}
\end{subequations}
where
\begin{align*}
(N^H)^i _j (\Hdot, \vdot) &= - \vdot ^p \deep \hdot {i}{j} +- \hdot {p}{j} \deep \vdot ^i
\\
(N^v)^i (\Hdot, \vdot) &= -  \vdot ^p \deep \vdot ^i -(\ahat{lm}{pj} (H) \hdot{p}{i} \deel \hdot{j}{m} + [\hat{A} (H)- \hat{A} (I)]^{lm}_{ij} \deel \hdot{j}{m}) \\
(M^H)^i (\Hdot) &= - \deei [\frac{1}{2} ((\tr \Hdot)^2-\tr(\Hdot ^2))+\det \Hdot]
\end{align*}
with constraints (\ref{hdotconst1})-(\ref{hdotconst3}).

We also take this opportunity to write the system in the notation of \cite{ST2} with $\Udot=(\Hdot,\vdot)$, $A_k \Udot= (\vdot \otimes e_k, \Hdot e_k)$:
\begin{equation} \label{ueq}
L\Udot \equiv \deet \Udot - A_k\deek \Udot -(0,\visc \laplacian \vdot) = N(\Udot) -  (0, (c_1^2-1)M^H(\Hdot) + \grad p)
\end{equation}
where $N(\Udot)=(N^H(\Udot), N^v(\Udot))$.

\section{Notation and Preliminaries}
\setcounter{equation}{0}
\renewcommand{\theequation}{3.\arabic{equation}}

\normalsize
Throughout the paper we use the following notation without mention.  Most of our norms and inner products will be in $L^2$ and most integrals will be taken over $\rthree$, so we write
$$\norm{\cdot} = \ltwonorm{\cdot},$$
$$\innerprod{\cdot}{\cdot} = \ltwoip{\cdot}{\cdot},$$
and
$$\int =\int _{\mathbf{R}^3}.$$
The notation
$$\langle c \rangle = (c^2 +1)^{\frac{1}{2}}$$
is used as a substitute for the real number $c$ when being bounded away from zero is important.

We use the usual derivative vector fields
$$\partial = (\partial _t, \partial _1, \partial _2, \partial _3) \quad \mbox{and} \quad \grad = (\partial _1, \partial _2, \partial _3)$$
as well as the scaling operator
$$S=t\deet +r\deer,$$
its time independent analogue
$$S_0 = r\deer,$$
and rotational derivatives
$$\Omega =(x \wedge \grad).$$
Here, as usual, $\deer = \frac{x}{r} \cdot \grad$.

Because we are working with vector valued functions we will also need to use
\begin{equation}\label{otwid1}
\omegatwiddle _i (U) = (\Omega _i H + [V^{(i)},H], \Omega _i v + V^{(i)} v),
\end{equation}
where
$$V^{(1)} = e_3\otimes e_2 - e_2\otimes e_3,$$
$$V^{(2)} = e_1\otimes e_3 - e_3\otimes e_1,$$
and
$$V^{(3)} = e_2\otimes e_1 - e_1\otimes e_2.$$
The notation $U=(H,v)$ and $\Udot = (\Hdot, \vdot)$ is used frequently, where $H,\Hdot$ are three by three real-valued matrix functions and $v,\vdot$ are three dimensional real valued vector functions.  The arguments of these functions are nearly always suppressed.  Occasionally we will write
\begin{equation} \label{otwid2}
\omegatwiddle H = \Omega H+ [V, H]~~\textrm{for}~H\in \rthree \otimes \rthree
\end{equation}
and
\begin{equation}\label{otwid3}
\omegatwiddle v = \Omega v + Vv~~\textrm{for}~v\in \rthree
\end{equation}
as notation for the components of $\omegatwiddle U$.  For scalar functions $f$ we define
\begin{equation}\label{otwid4}
\omegatwiddle f=\Omega f.
\end{equation}
For convenience we use the notation
$$\Upsilon = (\grad, \omegatwiddle ).$$
An exponent $\alpha$ on  $\Upsilon$ denotes an ordered $k$-tuple, $(\alpha_1,...,\alpha_k)\in \mathbf{Z}^k$, where $\Upsilon _{\alpha _i}$ makes sense for $1\leq i \leq k$.  We denote
$$\Upsilon ^\alpha = \Upsilon _{\alpha _1} ...\Upsilon _{\alpha _k}. $$
and
$$|\alpha | = k.$$
By $\Upsilon U$ we mean $\Upsilon _i U$ for some $i$.

The following identities and commutation properties of the vector fields are used as well. One extremely useful identity which we will use often is

\begin{equation}
\grad = \omega \deer - \frac{1}{r} (\omega \wedge \Omega)
\label{polargrad}
\end{equation}
for $\omega=\frac{x}{r}$.  The commutator of any two $\Upsilon$s is a $\Upsilon$.  We note that $\Upsilon$ is a Lie algebra with the Lie bracket given by the commutator operation
$$[\Upsilon_i,\Upsilon_j]=\Upsilon_i\Upsilon_j-\Upsilon_j\Upsilon_i.$$
We denote
\begin{equation}
A(\grad) U = A_k \partial_k U = ({ \grad v} , { \grad \cdot H}) \label{defa}
\end{equation}
and recall the definition (\ref{ueq})
\begin{equation}
LU = \deet U + A(\grad) U +(0, -\visc \laplacian v).
\label{defl}
\end{equation}
Because of our assumptions of objectivity and isotropy, $L$ commutes with $\Upsilon$, that is
\begin{equation}
\Upsilon ^\alpha LU = L\Upsilon ^\alpha U.
\label{oopscommute}
\end{equation}
However, $S$ does not commute with $L$.  For a function $f$ and $n=1,2,3,...$ we have the identities

\begin{equation}
S^n \partial f = \partial (S-1)^n f
\label{scommute}
\end{equation}
and

\begin{equation}
S^n \laplacian f = \laplacian \sum _{j=0} ^n
{n\choose j}
(-1)^{n-j}(S-1)^j  f.
\label{snlaplcommute}
\end{equation}
So, if we use the notation $(S-1)=\Stwiddle$ we have

\begin{equation}
S^n LU = [\deet + A(\grad )]\Stwiddle ^n U
-\visc \laplacian \sum _{j=0} ^n
{n\choose j}
(-1)^{n-j}\Stwiddle ^j U.
\label{sncommute}
\end{equation}
One of our main concerns will be dealing with this commutator.

\section{Spaces, Norms and Statement of Main Theorems}

\setcounter{equation}{0}
\renewcommand{\theequation}{4.\arabic{equation}}

Writing $S_0 = r\deer$, the space for the initial conditions is
$$H^{\sigma , \theta} _\Lambda =\{ U=(H,v) : \mathbf{R}^3 \rightarrow (\rthree \otimes \rthree) \times \rthree ~|~ S_0 ^a \Upsilon ^\alpha U \in L^2(\mathbf{R}^3), |\alpha | +a \leq \sigma, a\leq \theta \}.$$
with norm
$$\norm{U} _{H^{\sigma , \theta} _\Lambda} ^2 =  \sum _{|\alpha | +a \leq \sigma \atop a\leq \theta} \norm{S_0 ^a \Upsilon ^\alpha U} ^2.$$
Solutions will be constructed in the space
\begin{align*}
H^{\sigma , \theta} _\Gamma  = \{ U&=(H,v):[0,\infty )\times \rthree \rightarrow (\rthree \otimes \rthree) \times \rthree | \\
& \stwiddleoops{a}{\alpha} \Udot \in C([0,\infty ) ; L^2(\rthree)) ~\textrm{for}~ a\leq \theta, a+|\alpha | \leq \sigma  \}.
\end{align*}
Define corresponding norm to be
$$\norm{U}_{H^{\sigma , \theta} _\Gamma } = \sup_{0\leq t} \Big \{ \langle t \rangle^{-\sqrt{\delta}} \sum _{|\alpha | +a \leq \sigma \atop a\leq \theta} \norm{\stwiddleoops{a}{\alpha} \Udot} \Big \}$$
where $\delta <1$ as in Theorem \ref{main} below.
Given $U \in H^{\sigma , \theta} _\Gamma $ define the energy by
\begin{equation}
\E{\sigma}{\theta} [U] =\frac{1}{2}\sum _{{|\alpha | +a \leq \sigma} \atop {a\leq \theta}}  \int _\rthree \left ( \ahat{lm}{pj} (H) \stwiddleoops{a}{\alpha} \hdot{p}{l} \stwiddleoops{a}{\alpha} \hdot{j}{m} + |\stwiddleoops{a}{\alpha} \vdot |^2 \right ). \label{energy}
\end{equation}
As long as $|\Udot |$ is small, as will always be assumed, we have
\begin{equation} \label{normen}\E{\sigma}{\theta} [U] \sim \sum _{{|\alpha | +a \leq \sigma} \atop {a\leq \theta}} \norm{\stwiddleoops{a}{\alpha} \Udot}^2
\end{equation}
via (\ref{propa}) and the standard Sobolev Embedding Theorem in three dimensions.  This equivalence will be used repeatedly in the sequel without further comment.  The time independent analogue of $E$ is
\begin{equation}
\mathcal{E}_{\sigma, \theta} [U] =\frac{1}{2}\sum _{{|\alpha | +a \leq \sigma} \atop {a\leq \theta}}  \int _\rthree \left ( \ahat{lm}{pj} (H) S_0^{a} \Upsilon^{\alpha} \hdot{p}{l} S_0^{a} \Upsilon^{\alpha} \hdot{j}{m} + |S_0^{a} \Upsilon^{\alpha} \vdot |^2 \right ).
\end{equation}

Using these definitions, we state the main result.

\begin{theorem} \label{main}
Let $X_0$ be an orientation and volume preserving diffeomorphism on $\rthree$, and let $v_0$ be a divergence free vector field on $\rthree$.  Define
\begin{align*}
& U_0 = (H_0,v_0) = (\grad X_0, v_0) \\
& \Udot _0 = (\Hdot _0, \vdot _0) = (H_0-I, v_0 ).
\end{align*}
Suppose that $\Udot _0 \in H^{\kappa , \kappa -3} _\Lambda$ with $\kappa \geq 17$ and that
$$\mathcal{E}_{\kappa, \kappa -3} [U_0] < C, ~~\mathcal{E}_{\kappa-4, \kappa-4} [U_0] < \varepsilon $$
for uniform constants $C$ and $\varepsilon$.

If $\varepsilon$ is sufficiently small, then the initial value problem for (\ref{heq1})-(\ref{heq2}) with initial data $U_0 = U(0)$ has a unique solution $U(t)$ which satisfies the constraints (\ref{hconst1})-(\ref{hconst3}) and is a member of $ H^{\kappa , \kappa -3} _\Gamma  $.  Furthermore, the magnitudes of $\mathcal{E}_{\kappa, \kappa -3} [U_0]$ and $\mathcal{E}_{\kappa-4, \kappa-4} [U_0]$ do not depend on the size of the parameter $\visc$ and the solution satisfies the estimates

\begin{align}
&  \E{\kappa-4}{\kappa-4} [U(t)] + \frac{\visc}{2} \sum _{{a+|\alpha| \leq \kappa-4}} \int_0^t \norm{ \grad \stwiddleoops{a}{\alpha} v} ^2 d\tau \leq C' \E{\kappa-4}{\kappa-4} [U_0] \leq C'\varepsilon \label{main1}\\
& \E{\kappa}{\kappa -3} [U(t)] + \frac{\visc}{2} \sum _{{a+|\alpha| \leq \kappa} \atop {a\leq \kappa-3}} \int_0^t \norm{ \grad \stwiddleoops{a}{\alpha} v} ^2 d\tau \leq C' \E{\kappa}{\kappa -3} [U_0] \langle t \rangle^{ \delta} \label{main2}
\end{align}
for all $t\in(0,\infty)$, where $C'$ is a uniform constant and $\delta < 1$.
\end{theorem}
The proof of this result, which will be given in section 10, uses local energy decay estimates, Sobolev inequalities and hyperbolic energy methods.

We remark that these estimates reduce to those of 3D incompressible elastodynamics (see \cite{ST2}) as $\nu \rightarrow 0$.

\section{Local Energy Decay}

\setcounter{equation}{0}
\renewcommand{\theequation}{5.\arabic{equation}}

In order to take full advantage of energy estimates, we need to derive inequalities that establish decay away from characteristic cones for solutions of the system (\ref{hdoteqc2a})-(\ref{hdoteqc2b}) with constraints (\ref{hdotconst1})-(\ref{hdotconst3}).  The isotropy of the system implies all of the necessary hypotheses for local energy decay which are discussed at length in \cite{ST2}.  The subtle new feature of our equations is appearance of viscosity term only in equation (\ref{hdoteq2}).  When $\visc = 0$ or if we artificially add a viscosity term, $-\visc \laplacian \Hdot$, to (\ref{hdoteq1}) the following result reduces to \cite{ST2}.

Before we begin the proving the estimates, we must specify the regions in $\mathbf{R} \times \mathbf{R}^3$ where the estimates will hold. These regions are dictated by the light cone $|x|=t$ associated with the linear problem.   Define a smooth cutoff $\zeta : \mathbf{R}\rightarrow \mathbf{R}$ which satisfies
$$\zeta \equiv 1 ~~\textrm{on} ~[0,1]$$
and
$$\zeta \equiv 0 ~~\textrm{on} ~[2,\infty ).$$
Let $m> 4$ be a fixed integer.  Define
\begin{equation}
\eta (t,r) = \zeta \left (\frac{mr}{\ctwot} \right ) \label{eta}
\end{equation}
and
\begin{equation}
\gamma (t,r)= 1-\eta (t,r). \label{gamma}
\end{equation}
Abusing notation, we will write $\eta '= \zeta ' (\frac{mr}{\ctwot} )$ and $\gamma '=-\eta '$.

\vspace{.3in}

\begin{theorem} \label{ledthm}  Suppose $\Udot = (\Hdot, \vdot)$  is a solution of

\begin{equation}
\begin{array}{c}
\deet \Hdot+  \grad \vdot = f \\
\\
\deet \vdot +  \grad \cdot \Hdot -\visc \laplacian \vdot = g
\\
\\
\end{array}
\label{hdoteqfg}
\end{equation}
where $f$ and $g$ are forcing terms and that $\Udot$ also satisfies the constraints (\ref{hdotconst1}) and (\ref{hdotconst3}). Then the following bounds hold for $n=1,2,3...$ and exponent $\alpha=(\alpha_1,...,\alpha_k)$ whenever the right hand sides make sense:

\begin{align}
\displaystyle{\sum _{a=0} ^{n}} & \{ \norm{\eta \ctwot \grad \stwiddleoops{a}{\alpha} \Udot} + \norm{\eta \visc t\laplacian \stwiddleoops{a}{\alpha} \vdot} \} \label{intled} \\
& \lesssim \sum _{a=0} ^{n}
\{ \visc \norm{ \grad \stwiddleoops {a+1}{\alpha} \Hdot}
+ \norm{ \stwiddleoops {a+1}{\alpha} \Udot}
+ \norm{\grad \stwiddleoops {a}{\alpha} \Udot}  \notag \\
& \hspace{.2in} + \norm{ \stwiddleoops {a}{\alpha} \Udot}
+ \norm{ \eta t \soops {a}{\alpha} f}
+ \norm{ \eta t \soops {a}{\alpha} g}
+ \visc \norm{ \eta t \grad \cdot \soops {a}{\alpha} f} \}.  \notag
\end{align}
\vspace{.1in}
\begin{align}
\! \displaystyle{\sum _{a=0} ^{n}} & \{ \norm{\gamma [r \deer - t A(\grad )]\stwiddleoops{a}{\alpha} \Udot } + \norm{\gamma \visc t\laplacian \stwiddleoops{a}{\alpha} \vdot} \} \label{extled} \\
& \lesssim \sum _{a=0} ^{n}
\{ \norm{ \omegatwiddle \stwiddleoops {a}{\alpha} \Hdot}
+ \norm{ \stwiddleoops {a+1}{\alpha} \Udot}
+ \norm{\grad \stwiddleoops {a}{\alpha} \vdot}  \notag \\
& \hspace{.2in} + \norm{ \stwiddleoops {a}{\alpha} \vdot}
+ \norm{ \gamma t \soops {a}{\alpha} f}
+ \norm{ \gamma t \soops {a}{\alpha} g} \}.  \notag
\end{align}
Furthermore, the suppressed constants do not depend on the viscosity $\nu$.

\end{theorem}

\begin{proof} If we take derivatives $\soops{a}{\alpha}$ in equations (\ref{hdoteqfg}), use commutation properties of the vector fields (\ref{oopscommute}) and (\ref{sncommute}), we get
\begin{subequations}
\begin{align}
& \deet \Htwiddle +  \grad \vtwiddle=\ftwiddle \label{led1a}\\
& \deet \vtwiddle +   \grad \cdot \Htwiddle -\visc  \laplacian \vtwiddle
=\gtwiddle \label{led1b}
\end{align}
\end{subequations}
where $(\Htwiddle, \vtwiddle)=\Utwiddle=\stwiddleoops{a}{\alpha}\Udot$, $\ftwiddle = \soops{a}{\alpha} f$ and

\begin{equation} \label{gtwiddle} \gtwiddle = \soops{a}{\alpha} g +\visc \laplacian \sum _{j=0} ^{a-1}
{n\choose j}
(-1)^{a-j}\Stwiddle ^j \Upsilon ^\alpha \vdot.
\end{equation}
Multiplying equations (\ref{led1a})-(\ref{led1b}) by $t$, moving the $t\deet$ terms to the right and rewriting them using the scaling operator gives
\begin{subequations}
\begin{align}
&  t \grad \vtwiddle=r\deer \Htwiddle -S\Htwiddle + t\ftwiddle \label{led2}\\
&  t \grad \cdot \Htwiddle -\visc t \laplacian \vtwiddle
=r\deer \vtwiddle -S\vtwiddle +t\gtwiddle. \label{led3}
\end{align}
\end{subequations}
Equations (\ref{led2})-(\ref{led3}) will be our starting point for the derivation of the estimates (\ref{intled}) and (\ref{extled}).

\vspace{.1in}

\noindent \textit{Interior Estimate.} $(r<\frac{\ctwot}{m})$

To begin we multiply equations (\ref{led2})-(\ref{led3}) by $\eta$, apply $\norm{\cdot} ^2$ to both equations and use the triangle inequality to obtain
\begin{align}
\norm{\eta  t \grad \vtwiddle}^2
&\leq 2\norm{\eta r\deer \Htwiddle}^2
+4\norm{\eta S\Htwiddle}^2
+ 4\norm{\eta t\ftwiddle}^2 \label{int1}\\
\notag \\
\norm{\eta  t \grad \cdot \Htwiddle}&^2
+ \norm{\eta \visc t \laplacian \vtwiddle}^2
-2\innerprod{\eta  t \grad \cdot \Htwiddle} {\eta \visc t \laplacian \vtwiddle} \label{int2} \\
&\leq 2\norm{\eta r\deer \vtwiddle}^2
+4\norm{\eta S\vtwiddle}^2
+4\norm{\eta t\gtwiddle}^2. \notag
\end{align}
where above and in the future we use the notation $\innerprod{\cdot}{\cdot} = \innerprod{\cdot}{\cdot}_{L^2(\rthree)}$.
The main calculation is estimating the cross term in (\ref{int2}).  Taking divergence of (\ref{led2}), using commutation properties, and multiplying by $\eta  \visc$ we get
\begin{align}
\eta \visc t \laplacian \vtwiddle =
\eta  \visc (r\deer \grad \cdot \Htwiddle -\grad \cdot \Stwiddle \Htwiddle +t\grad \cdot \ftwiddle).
\label{int3}
\end{align}
Thus,
\begin{align}
-2 & \innerprod{\eta  t \grad \cdot \Htwiddle} {\eta \visc t \laplacian \vtwiddle} \label{int4} \\
&=-2\innerprod{\eta  t \grad \cdot \Htwiddle} {\eta  \visc (r\deer \grad \cdot \Htwiddle -\grad \cdot \Stwiddle \Htwiddle +t\grad \cdot \ftwiddle)} \notag \\
&=-2\innerprod{\eta  t \grad \cdot \Htwiddle} {\eta \visc r\deer \grad \cdot \Htwiddle}
+2\innerprod{\eta  t \grad \cdot \Htwiddle} {\eta \visc  \grad \cdot \Stwiddle \Htwiddle} \notag \\
& \hspace{1in} -2\innerprod{\eta  t \grad \cdot \Htwiddle} {\eta  \visc \grad \cdot \ftwiddle} \notag\\
& \equiv (I1) +(I2) +(I3). \notag
\end{align}
For any $\yconst =1,2,3...$,
\begin{align}
(I2) \geq -\yconst \visc ^2 \norm{\eta \grad \cdot \Stwiddle \Htwiddle}^2 -\frac{1}{\yconst}\norm {\eta  t \grad \cdot \Htwiddle}^2 \notag
\end{align}
and
\begin{align}
(I3) \geq -\yconst \visc ^2\norm{\eta t \grad \cdot  \ftwiddle}^2 -\frac{1}{\yconst}\norm {\eta  t \grad \cdot \Htwiddle}^2 \notag
\end{align}
by Young's Inequality.

\vspace{.2in}

Estimating $(I1)$ is a more significant calculation. For all $\yconst = 1,2,3...$
\begin{align}
(I1) &=-2\visc t\int \eta ^2 \deek \Htwiddle ^i_k x_j\deej \deep \Htwiddle ^i_p \notag \\
&= -2\visc t \int \eta ^2 x_j \deej (\frac{1}{2}  |\grad \cdot \Htwiddle|^2) \notag \\
&= 2 \visc t \int \eta \eta ' \frac{mx_j}{\ctwot r} x_j |\grad \cdot \Htwiddle|^2 + 3\visc t \int \eta ^2 |\grad \cdot \Htwiddle|^2 \notag \\
&\geq -2\visc t \int \eta |\eta '| \frac{mr}{\ctwot} |\grad \cdot \Htwiddle|^2 \notag \\
&\geq -4\visc t \int \eta \max |\eta '| |\grad \cdot \Htwiddle|^2 \notag \\
&= -2\int (2\visc \max |\eta '| |\grad \cdot \Htwiddle|) (\eta t |\grad \cdot \Htwiddle|) \notag \\
&\geq -\frac{1}{\yconst}\norm {\eta  t \grad \cdot \Htwiddle}^2 -4 \yconst \visc ^2\max{|\eta '|}^2 \norm {\grad \cdot \Htwiddle}^2. \notag
\end{align}
Here we used integration by parts,$ \frac{mr}{\ctwot} \leq 2$ on supp $\eta$, and Young's Inequality.  The value of $\yconst$ will be determined later. Adding estimates, we have
\begin{align}
-2 \innerprod{\eta  t \grad \cdot \Htwiddle} {\eta \visc t \laplacian \vtwiddle}
&=(I1) +(I2) +(I3) \label{int5} \\
&\geq -\frac{3}{\yconst}\norm {\eta  t \grad \cdot \Htwiddle}^2 -4 \yconst \visc ^2 \max{|\eta '|}^2 \norm {\grad \cdot \Htwiddle}^2
 \notag \\
& \hspace{.5in}-\yconst \visc ^2\norm{\eta \grad \cdot \Stwiddle \Htwiddle}^2
-\yconst \visc ^2 \norm{\eta t \grad \cdot  \ftwiddle}^2. \notag
\end{align}
Combining (\ref{int2}) and (\ref{int5}) gives
\begin{align}
(1&-\frac{3}{\yconst}) \norm{\eta  t \grad \cdot \Htwiddle}^2
+ \norm{\eta \visc t \laplacian \vtwiddle}^2 \label{int6} \\
&\leq 2\norm{\eta r\deer \vtwiddle}^2
+4\norm{\eta S\vtwiddle}^2
+4\norm{\eta t\gtwiddle}^2
+4 \yconst \visc ^2 \max{|\eta '|}^2 \norm {\grad \cdot \Htwiddle}^2 \notag \\
& \hspace{.5in} +\yconst \visc ^2 \norm{\eta \grad \cdot \Stwiddle \Htwiddle}^2
+\yconst \visc ^2 \norm{\eta t \grad \cdot  \ftwiddle}^2. \notag
\end{align}
Now we compute
\begin{align}
\norm{\eta  t \grad \cdot \Htwiddle}^2 &= \int \eta ^2  t^2 \deek \Htwiddle ^i _k \deep \Htwiddle ^i_p \label{int7}\\
&=-\int 2 \eta \eta '  t^2 \frac{mx_k}{\ctwot r} \Htwiddle ^i _k \deep \Htwiddle ^i_p -\int \eta ^2  t^2  \Htwiddle ^i _k \deek \deep \Htwiddle ^i_p \notag \\
&\geq-2\int \eta  t m |\eta '| |\Htwiddle ||\grad \cdot \Htwiddle |
+\int 2 \eta \eta ' t^2 \frac{mx_p}{\ctwot r} \Htwiddle ^i _k \deek \Htwiddle ^i_p \notag \\
& \hspace{.5in} +\int \eta ^2  t^2 \deep \Htwiddle ^i _k \deek \Htwiddle ^i_p  \notag \\
&\geq -2m\int |\eta '||\Htwiddle |(\eta t(|\grad \Htwiddle |+|\grad \cdot \Htwiddle |))
+\int \eta ^2  t^2 |\grad \Htwiddle |^2 \notag \\
&\geq -9m^2\max|\eta '|^2 \norm{\Htwiddle}^2 -\frac{1}{9}(\norm{\eta t\grad \Htwiddle}^2 +\norm{\eta t\grad \cdot \Htwiddle}^2)\notag \\
& \hspace{.5in}+\norm{\eta t\grad \Htwiddle}^2 \notag
\end{align}
where we used integration by parts twice, constraint (\ref{hdotconst3}), and Young's Inequality. Estimate (\ref{int7}) implies
\begin{align}
\frac{4}{5} \norm{\eta t\grad \Htwiddle}^2 - 9m^2\max|\eta '|^2 \norm{\Htwiddle}^2 \leq \norm{\eta t\grad \cdot \Htwiddle}^2. \label{int8}
\end{align}
If we combine (\ref{int1}), (\ref{int6}) and (\ref{int8}), we have
\begin{align}
(1&-\frac{3}{\yconst}) \frac{4}{5} \norm{\eta t\grad \Utwiddle}^2
+\norm{\eta \visc t \laplacian \vtwiddle}^2 \label{int9} \\
&\leq 2\norm{\eta r\deer \Utwiddle}^2
+4\norm{\eta S\Utwiddle}^2
+4\norm{\eta t\gtwiddle}^2
+4\norm{\eta t\ftwiddle}^2
+9m^2\max|\eta '|^2 \norm{\Htwiddle}^2 \notag \\
&\hspace{.2in}+4 \yconst \visc ^2 \max{|\eta '|}^2 \norm {\grad \cdot \Htwiddle}^2
+\yconst \visc ^2 \norm{\eta \grad \cdot \Stwiddle \Htwiddle}^2
+\yconst \visc ^2 \norm{\eta t \grad \cdot  \ftwiddle}^2.\notag
\end{align}

\vspace{.2in}
Using $\norm{\eta r\deer \Utwiddle}^2 \leq \norm{\eta \frac{2 \ctwot}{m} \deer \Utwiddle}^2 \leq \frac{4}{m^2}\norm{\eta t\grad \Utwiddle}^2$, and taking $\yconst =8$, we can absorb $2\norm{\eta r\deer \Utwiddle}^2$ on the left:
\begin{align}
(\frac{1}{2} - & \frac{8}{m^2})\norm{\eta \ctwot \grad \Utwiddle}^2 +\norm{\eta \visc t \laplacian \vtwiddle}^2 \label{int10}\\
&\leq 4\norm{\eta S\Utwiddle}^2
+4\norm{\eta t\gtwiddle}^2
+4\norm{\eta t\ftwiddle}^2
+8 \visc ^2\norm{\eta t \grad \cdot  \ftwiddle}^2
+\norm{\eta \grad \Utwiddle} ^2 \notag \\
&\hspace{.2in} +9m^2\max|\eta '|^2 \norm{\Htwiddle}^2
+ 32 \visc ^2 \max{|\eta '|}^2 \norm {\grad \cdot \Htwiddle}^2
+8 \visc ^2 \norm{\eta \grad \cdot \Stwiddle \Htwiddle}^2
\notag
\end{align}
which we can write without explicit constants as
\begin{align}
&  \norm{\eta \ctwot \grad \stwiddleoops{a}{\alpha} \Udot} + \norm{\eta \visc t\laplacian \stwiddleoops{a}{\alpha} \vdot}  \label{int11} \\
& \lesssim
 \visc \norm{ \grad \stwiddleoops {a+1}{\alpha} \Hdot}
+ \norm{ \stwiddleoops {a+1}{\alpha} \Udot}
+ \norm{\grad \stwiddleoops {a}{\alpha} \Udot}
+ \norm{ \stwiddleoops {a}{\alpha} \Udot} \notag \\
& + \norm{ \eta t \soops {a}{\alpha} f} 
+ \norm{ \eta t \soops {a}{\alpha} g}
+ \visc \norm{ \eta t \grad \cdot \soops {a}{\alpha} f} \notag \\
& + \norm{ \eta \visc t \laplacian \sum _{j=0} ^{a-1}
{n\choose j}  
(-1)^{a-j}\Stwiddle ^j \Upsilon ^\alpha \vdot }   \notag
\end{align}
after recalling the definitions of $\Utwiddle, \Htwiddle, \vtwiddle, \ftwiddle,$ and $\gtwiddle$.  We pause here to derive the corresponding exterior estimate.

\vspace{.2in}

\noindent \textit{Exterior Estimate.} $(r>\frac{2\ctwot}{m})$

Rearranging (\ref{led2})-(\ref{led3}), multiplying by $\gamma$, applying $\norm{\cdot} ^2$ to both equations, using the triangle inequality and adding the resulting estimates together yields

\begin{align}
\norm{\gamma r\deer \Utwiddle & - \gamma t A(\grad) \Utwiddle}^2
+\norm{\gamma t\visc \laplacian \vtwiddle}^2
-2\innerprod{\gamma  t \grad \cdot \Htwiddle-\gamma r\deer \vtwiddle}{\gamma \visc t \laplacian \vtwiddle} \label{ext1} \\
&\leq 2\norm{\gamma S\Utwiddle}^2
+2\norm{\gamma t\ftwiddle} ^2
+2\norm{\gamma t\gtwiddle}^2 \notag
\end{align}
Again, most of the work is calculating the cross terms.  The simpler of the two is

\begin{align}
2\innerprod{\gamma r\deer \vtwiddle}{\gamma \visc t \laplacian \vtwiddle}
&=2\visc t \int \gamma ^2 r\deer \vtwiddle ^i \laplacian \vtwiddle \label{ext2} \\
&=-4 \visc t \int \gamma \gamma ' \frac{x_km}{\ctwot r} x_j \deej \vtwiddle ^i \deek \vtwiddle ^i
-2\visc t \int \gamma ^2 \deek \vtwiddle^i \deek \vtwiddle^i \notag \\
& \hspace{.5in} -2\visc t \int \gamma ^2 x_j \deej \deek \vtwiddle ^i \deek \vtwiddle ^i
\notag \\
&\geq -4\visc t \int \gamma |\gamma '| \frac{mr^2}{\ctwot r} |\deer \vtwiddle |^2
-2\visc t\int \gamma ^2 |\grad \vtwiddle |^2 \notag \\
& \hspace{.2in}+2\visc t \int \gamma \gamma '  \frac{x_jm}{\ctwot r} x_j \deek \vtwiddle ^i \deek \vtwiddle ^i +3\visc t \int \gamma ^2 |\grad \vtwiddle |^2 \notag \\
& \geq -8\visc t\int \gamma |\gamma '| |\deer \vtwiddle |^2 -4\visc t \int \gamma |\gamma '| \frac{mr^2}{\ctwot r} |\grad \vtwiddle |^2 \notag \\
& \hspace{.5in} +\visc t \int \gamma ^2 |\grad \vtwiddle |^2. \notag \\
&\geq -12\visc t \max |\gamma '| \int \gamma  |\grad \vtwiddle|^2, \notag
\end{align}
where we used integration by parts, and $\frac{mr}{\ctwot} \geq 2$ on supp $\gamma '$. We pause to integrate by parts

\begin{align}
-\int \gamma  |\grad \vtwiddle|^2
& = -\int  \gamma \deek \vtwiddle ^i \deek \vtwiddle ^i \label{ext3} \\
&=  -\int \gamma ' \frac{x_km}{\ctwot r} \vtwiddle ^i \deek v^i - \int \gamma \vtwiddle ^i \laplacian \vtwiddle ^i \notag \\
&\geq  -\int |\gamma '| \frac{m}{\ctwot} |\vtwiddle ||\grad \vtwiddle | - \int \gamma |\vtwiddle ||\laplacian \vtwiddle |  \notag \\
&\geq  -m \ctwot ^{-1} \max |\gamma '|  \int |\vtwiddle ||\grad \vtwiddle | -\int |\vtwiddle | |\gamma \laplacian \vtwiddle |. \notag
\end{align}
Now (\ref{ext2}), (\ref{ext3}) and Young's inequality imply

\begin{align}
2\innerprod{\gamma r\deer \vtwiddle}{\gamma \visc t \laplacian \vtwiddle}
&\geq -12\visc t \max |\gamma '| \int \gamma  |\grad \vtwiddle|^2 \label{ext4} \\
&\geq -12\visc \max|\gamma '|^2 \int |\vtwiddle ||\grad \vtwiddle |
 -12 \max |\gamma '| \int |\vtwiddle | |\gamma \visc t \laplacian \vtwiddle | \notag \\
&\geq -\max|\gamma '|^2 \left ( 72\yconst \norm{\vtwiddle}^2
+\frac{(\visc m)^2}{ \yconst} \norm{\grad \vtwiddle}^2 \right ) -\frac{1}{\yconst} \norm{\gamma \visc t \laplacian \vtwiddle}^2 ,\notag
\end{align}
for any positive integer $\yconst$.

To estimate the  other cross term in (\ref{ext1}), we start by recalling (\ref{polargrad})
 \[
 \grad = \omega \deer - \frac{1}{r} (\omega \wedge \Omega)
 \]
to get
\begin{align}
  \grad \cdot \Htwiddle &=  [\omega \deer - \frac{1}{r} (\omega \wedge \Omega)]\cdot \Htwiddle \label{ext5} \\
 &= \omega \cdot \deer \Htwiddle - \frac{1}{r} (\omega \wedge \Omega)\cdot \Htwiddle. \notag
\end{align}
Using (\ref{led2}) we have

\begin{align}
  t \omega \cdot \deer \Htwiddle &=   t \omega \cdot (\frac{ t}{r} \grad \vtwiddle + \frac{1}{r} S\Htwiddle - \frac{t}{r} \ftwiddle) \label{ext6} \\
& =  \frac{t^2}{r} \omega \cdot \grad \vtwiddle  +  \frac{ t} {r} \omega \cdot S\Htwiddle - \frac{ t^2} {r} \omega \cdot \ftwiddle. \notag
\end{align}
Combining (\ref{ext5}), and (\ref{ext6}) gives us

\begin{align}
-2\innerprod{\gamma  & t \grad \cdot \Htwiddle}{\gamma \visc t \laplacian \vtwiddle} \label{ext7} \\
&= -2\innerprod{\gamma \frac{t^2}{r} \omega \cdot \grad \vtwiddle}{\gamma \visc t \laplacian \vtwiddle}
-2\innerprod{\gamma \frac{ t} {r} \omega \cdot S\Htwiddle }{\gamma \visc t \laplacian \vtwiddle} \notag \\
& \hspace{.2in} +2\innerprod{\gamma \frac{ t^2} {r} \omega \cdot \ftwiddle}{\gamma \visc t \laplacian \vtwiddle}
+2 \innerprod{\gamma  t \frac{1}{r} (\omega \wedge \Omega)\cdot \Htwiddle}{\gamma \visc t \laplacian \vtwiddle} \notag \\
&\equiv (E1)+(E2)+(E3)+(E4) .\notag
\end{align}
Using Young's inequality and the fact that $\frac{\ctwot }{mr}\leq 1$ on supp $\gamma$, we can bound $(E2)+(E3)+(E4)$ below by

\begin{align}
-\frac{3}{\yconst} \norm{\gamma \visc t \laplacian \vtwiddle}^2 -\yconst m^2 (\norm{\gamma S\Htwiddle }^2 + \norm{\gamma t\ftwiddle}^2 + \norm{\gamma \Omega \Htwiddle}^2) \label{ext8}
\end{align}
Now we only have one term left to bound:

\begin{align}
(E1) &= -2\visc  t^3 \int \gamma ^2 r^{-2} x_k \deek \vtwiddle ^i \laplacian \vtwiddle ^i \label{ext9}
\\
&= 2\visc  t^3 \Big( \int \gamma ^2 r^{-2} x_k \deek \deej \vtwiddle ^i \deej \vtwiddle ^i +\int 2\gamma \gamma ' \frac{m x_k x_j}{\ctwot r^3} \deek \vtwiddle ^i \deej \vtwiddle ^i \notag
\\
&\hspace{.5in} + \int \gamma ^2 r^{-2} \deej \vtwiddle ^i \deej \vtwiddle ^i -\int 2\gamma ^2 \frac{x_k x_j}{r^4} \deek \vtwiddle ^i \deej \vtwiddle ^i \Big)\notag
\\
&= 2\visc  t^3 \Big( \int \gamma ^2 r^{-2} x_k \deek (\frac{1}{2} |\grad \vtwiddle|^2)
+\int \gamma ^2 r^{-2} |\grad \vtwiddle|^2 \Big) \notag
\\
&\hspace{.5in} +4\visc  t^3 \Big( \int \gamma \gamma ' \frac{m}{\ctwot r} |\deer \vtwiddle |^2 - \int \gamma ^2 r^{-2} |\deer \vtwiddle |^2 \Big)
\notag \\
&\geq 2\visc  t^3\Big( -\int \gamma \gamma ' \frac{mx_k}{\ctwot r^3} x_k |\grad \vtwiddle|^2
+\int \gamma ^2 \frac{x_k}{ r^4} x_k |\grad \vtwiddle|^2
\notag \\
& \hspace{.2in} -\frac{3}{2} \int \gamma ^2 r^{-2} |\grad \vtwiddle|^2 -\int \gamma ^2 r^{-2} |\deer \vtwiddle |^2 - 2 \int \gamma |\gamma '| \frac{m}{\ctwot r} |\deer \vtwiddle |^2 \Big).
\notag
\end{align}
So far, we have only used integration by parts.  Continuing our estimate we repeatedly apply the facts $\frac{mr}{\ctwot }\leq 1$ on supp $\gamma$ and $|\deer \vtwiddle | \leq |\grad \vtwiddle |$ to get
\begin{align}
(E1) &\geq -6\visc  t^3 m^2 \Big( \int \gamma |\gamma '| \ctwot ^{-1} (mr)^{-1} |\grad \vtwiddle|^2
+ \int \gamma ^2 (mr)^{-2} |\grad \vtwiddle|^2 \Big)
 \label{ext9a}\\
&\hspace{.5in} -2\visc c_2^2 t^3 m^3 \int \gamma ^2 \ctwot ^{-1} (mr)^{-2} |\grad \vtwiddle|^2
\notag \\
& \geq -6\visc t m^2 \Big( \max  |\gamma '| \int \gamma |\grad \vtwiddle|^2
+\int \gamma ^2 |\grad \vtwiddle|^2 \Big)
-2\visc m^3 \int \gamma ^2 |\grad \vtwiddle|^2
\notag \\
& \geq -22\visc t m^2 \max|\gamma '| \int \gamma |\grad \vtwiddle|^2
\notag \\
&\geq -22m^3 \visc  \max |\gamma '| ^2 \int |\vtwiddle ||\grad \vtwiddle |
-22m^2 \max |\gamma '| \int |\vtwiddle | |\gamma \visc t \laplacian \vtwiddle |
\notag \\
&\geq -\max|\gamma '|^2 (242m^4\yconst \norm{\vtwiddle}^2
+\frac{(\visc m)^2}{\yconst} \norm{\grad \vtwiddle}^2)
-\frac{1}{\yconst} \norm{\gamma \visc t \laplacian \vtwiddle}^2 \notag
\end{align}
for any positive integer $\yconst$. In the last two lines we used (\ref{ext3}) and Young's inequality.

Combining (\ref{ext4}), (\ref{ext8}), and (\ref{ext9}) with $\yconst =10$, we have
\begin{align}
-2\innerprod{\gamma  t \grad & \cdot \Htwiddle-\gamma r\deer \vtwiddle}{\gamma \visc t \laplacian \vtwiddle} \label{ext10} \\
&\geq -\max |\gamma '|^2(720+2420m^4)\norm{\vtwiddle}^2
-\max |\gamma '|^2 \frac{(\visc m)^2}{5 }\norm{\grad \vtwiddle}^2 \notag \\
&\hspace{.2in}- 10m^2(\norm{\gamma S\Htwiddle }^2
+ \norm{\gamma t\ftwiddle}^2
+ \norm{\gamma \Omega \Htwiddle}^2)
-\frac{1}{2}\norm{\gamma \visc t \laplacian \vtwiddle}^2. \notag
\end{align}
Inserting this into (\ref{ext1}) gives

\begin{align}
\norm{\gamma r\deer \Utwiddle & - \gamma t A(\grad) \Utwiddle}^2
+\frac{1}{2} \norm{\gamma t\visc \laplacian \vtwiddle}^2
\label{ext11} \\
&\leq 2\norm{\gamma t\gtwiddle}^2
+(10m^2+2)(\norm{\gamma S\Utwiddle }^2
+ \norm{\gamma t\ftwiddle}^2
+\norm{\gamma \Omega \Htwiddle}^2) \notag \\
& \hspace{.2in} +\max |\gamma '|^2(720+2420m^4)\norm{\vtwiddle}^2
+\max |\gamma '|^2 \frac{(\visc m)^2}{5 } \norm{\grad \vtwiddle}^2
\notag
\end{align}
which can be written without explicit constants as

\begin{align}
\norm{\gamma r\deer & \stwiddleoops{a}{\alpha} \Udot  - \gamma t A(\grad) \stwiddleoops{a}{\alpha} \Udot}
+\norm{\gamma t\visc \laplacian \stwiddleoops{a}{\alpha} \vdot}
\label{ext12} \\
&\lesssim \norm{\gamma t\soops{a}{\alpha} g}
+\norm{\stwiddleoops{a+1}{\alpha} \Udot }
+ \norm{\gamma t \soops{a}{\alpha} f}
+\norm{\omegatwiddle \stwiddleoops{a}{\alpha} \Hdot} \notag \\
& \hspace{.2in} + \norm{\stwiddleoops{a}{\alpha} \Udot}
+\norm{\grad \stwiddleoops{a}{\alpha} \vdot}
+\norm{ \gamma \visc t \laplacian \sum _{j=0} ^{a-1}
{n \choose j}
(-1)^{a-j}\Stwiddle ^j \Upsilon ^\alpha \vdot }.
\notag
\end{align}

\noindent \textit{Induction on \textit{a}.}

Estimates (\ref{int11}) and (\ref{ext12}) are analogous, corresponding to (\ref{intled}) and (\ref{extled}) respectively.  For each, we must deal with the final term on the right hand side.  We will use an inductive argument to handle these terms.

To prove (\ref{intled}), first assume $a=0$. Using (\ref{int11}), we have
\begin{align}
\norm{\eta \ctwot \grad &\Upsilon ^{\alpha} \Udot} + \norm{\eta \visc t\laplacian \Upsilon ^{\alpha} \vdot}  \label{ind1} \\
& \lesssim
 \norm{ \grad \stwiddleoops {}{\alpha} \Hdot}
+ \norm{ \stwiddleoops {}{\alpha} \Udot}
+ \norm{\grad \Upsilon ^{\alpha} \Udot}
+ \norm{ \Upsilon ^{\alpha} \Udot}
+ \norm{ \eta t \Upsilon ^{\alpha} f} \notag \\
& \hspace{.2in}
+ \norm{ \eta t \Upsilon ^{\alpha} g}
+ \norm{ \eta t \grad \cdot \Upsilon ^{\alpha} f}.
\notag
\end{align}
For $a=1$
\begin{align}
&  \norm{\eta \ctwot \grad \stwiddleoops{}{\alpha} \Udot} + \norm{\eta \visc t\laplacian \stwiddleoops{}{\alpha} \vdot}  \label{ind2} \\
& \lesssim
 \norm{ \grad \stwiddleoops {2}{\alpha} \Hdot}
+ \norm{ \stwiddleoops {2}{\alpha} \Udot}
+ \norm{\grad \stwiddleoops {}{\alpha} \Udot}
+ \norm{ \stwiddleoops {}{\alpha} \Udot}
+ \norm{ \eta t \soops {}{\alpha} f} \notag \\
& \hspace{.2in}
+ \norm{ \eta t \soops {}{\alpha} g}
+ \norm{ \eta t \grad \cdot \soops {}{\alpha} f}
+ \norm{ \eta \visc t \laplacian \Upsilon ^\alpha \vdot }   \notag
\end{align}
which is bounded by
\begin{align}
& \displaystyle{\sum _{a=0} ^{1}} \{ \norm{ \grad \stwiddleoops {a+1}{\alpha} \Hdot}
+ \norm{ \stwiddleoops {a+1}{\alpha} \Udot}
+ \norm{\grad \stwiddleoops {a}{\alpha} \Udot}
+ \norm{ \stwiddleoops {a}{\alpha} \Udot} \label{ind3}  \\
& \hspace{.2in} + \norm{ \eta t \soops {a}{\alpha} f}
+ \norm{ \eta t \soops {a}{\alpha} g}
+ \norm{ \eta t \grad \cdot \soops {a}{\alpha} f} \} \notag
\end{align}
after applying (\ref{ind1}) to the final term of (\ref{ind2}).

More generally if we assume (\ref{intled}) holds for $n=k$, then we have
\begin{align}
& \norm{ \eta \visc t \laplacian \sum _{j=0} ^{k}
{n \choose j}
(-1)^{a-j}\Stwiddle ^j \Upsilon ^\alpha \vdot }  \label{ind4} \\
&  \lesssim
\sum _{j=0} ^{k} \norm{ \eta \visc t \laplacian
\Stwiddle ^j \Upsilon ^\alpha \vdot } \notag \\
&  \lesssim
\sum _{j=0} ^{k}
\{ \norm{ \grad \stwiddleoops {j+1}{\alpha} \Hdot}
+ \norm{ \stwiddleoops {j+1}{\alpha} \Udot}
+ \norm{\grad \stwiddleoops {j}{\alpha} \Udot}  \notag \\
& \hspace{.2in} + \norm{ \stwiddleoops {a}{\alpha} \Udot}
+ \norm{ \eta t \soops {j}{\alpha} f}
+ \norm{ \eta t \soops {j}{\alpha} g}
+ \norm{ \eta t \grad \cdot \soops {j}{\alpha} f} \}  \notag
\end{align}
which implies (using (\ref{int11}))

\begin{align}
& \norm{\eta \ctwot \grad \stwiddleoops{k+1}{\alpha} \Udot} + \norm{\eta \visc t\laplacian \stwiddleoops{k+1}{\alpha} \vdot} \label{ind5} \\
& \lesssim \sum _{a=0} ^{k+1}
\{ \norm{ \grad \stwiddleoops {a+1}{\alpha} \Hdot}
+ \norm{ \stwiddleoops {a+1}{\alpha} \Udot}
+ \norm{\grad \stwiddleoops {a}{\alpha} \Udot}  \notag \\
& \hspace{.2in} + \norm{ \stwiddleoops {a}{\alpha} \Udot}
+ \norm{ \eta t \soops {a}{\alpha} f}
+ \norm{ \eta t \soops {a}{\alpha} g}
+ \norm{ \eta t \grad \cdot \soops {a}{\alpha} f} \}  \notag
\end{align}
Combining this estimate with our inductive hypothesis gives (\ref{intled}) for general $n$.  (\ref{extled}) follows from an identical argument using (\ref{ext12}).

\end{proof}

\section{The Null Condition and Spectral Projections}

\setcounter{equation}{0}
\renewcommand{\theequation}{6.\arabic{equation}}

For $\omega \in S^2$ let  $\ev{+}=1,\ev{-}=-1$ and $\ev{0}=0$ be the eigenvalues of $A(\omega)\equiv A_k \omega_k$ with $A_k$ defined as in (\ref{ueq}). Denote the orthogonal projections onto the corresponding eigenspaces by $\sproj{+}(\omega),\sproj{-}(\omega)$ and $\sproj{0}(\omega)$.  Using the formula

\begin{equation}
\sproj{\projind}(\omega)\Udot = \prod _{\chi \neq \projind} \frac{1}{\lambda _\projind -\lambda _\chi} (A(\omega) - \lambda _\chi I) \Udot \label{projform}
\end{equation}
we compute
\begin{equation}
\sproj{+}(\omega)\Udot = \frac{1}{2}((\vdot+\Hdot \omega )\otimes \omega, \vdot+\Hdot \omega) \label{pplus}
\end{equation}
\begin{equation}
\sproj{-}(\omega)\Udot = \frac{1}{2}(-(\vdot-\Hdot \omega )\otimes \omega, \vdot - \Hdot \omega) \label{pminus}
\end{equation}
and
\begin{equation}
\sproj{0}(\omega)\Udot = (\Hdot (I-\omega \otimes \omega) , 0 ). \label{pzero}
\end{equation}

For notational simplicity we will suppress the argument $\omega$ in our calculations.  Using these projections we utilize a calculation that appeared in \cite{ST2}.
\begin{align}
|( \ev{\projind}t - r ) &\sproj{\projind} \deej \stwiddleoops{a}{\alpha} \Udot| \label{extproj}
\\
&=|\sproj{\projind}(tA(\omega)-rI) \deej \stwiddleoops{a}{\alpha} \Udot| \notag
\\
&\leq |(tA(\omega)-rI) \deej \stwiddleoops{a}{\alpha} \Udot| \notag
\\
&=|(tA_k-r\omega^kI)\omega^k \deej \stwiddleoops{a}{\alpha} \Udot| \notag
\\
&=|(tA_k-r\omega^kI)(\omega^j\deek + \frac{1}{r} \Omega_{kj})\stwiddleoops{a}{\alpha} \Udot| \notag
\\
&=\Big| \Big[\omega^j(tA(\grad)-r\deer)+(tA_k - r\omega^kI)\frac{1}{r} \Omega_{kj} \Big] \stwiddleoops{a}{\alpha} \Udot \Big|\notag
\\
&\leq |(tA(\grad)-r\deer) \stwiddleoops{a}{\alpha} \Udot | + C\Big|\frac{t}{r} +1\Big||\Omega \stwiddleoops{a}{\alpha} \Udot |. \notag
\end{align}
Now we rewrite estimate (\ref{extled}) as
\begin{align}
\! \sum_{j=1}^{3} \sum _\projind \displaystyle{\sum _{a=0} ^{n}} & \{ \norm {\gamma \langle \ev{\projind}t - r \rangle \sproj{\projind} \deej \stwiddleoops{a}{\alpha} \Udot } + \norm{\gamma \visc t\laplacian \stwiddleoops{a}{\alpha} \vdot} \} \label{extled1} \\
& \lesssim \sum _{a=0} ^{n}
\{ \norm{ \omegatwiddle \stwiddleoops {a}{\alpha} \Udot}
+ \norm{ \stwiddleoops {a+1}{\alpha} \Udot}
+ \norm{\grad \stwiddleoops {a}{\alpha} \vdot}  \notag \\
& \hspace{.2in} + \norm{ \stwiddleoops {a}{\alpha} \vdot}
+ \norm{ \gamma t \soops {a}{\alpha} f}
+ \norm{ \gamma t \soops {a}{\alpha} g} \}  \notag
\end{align}
where $\omega=\frac{x}{|x|}$ in each expression above.

The null condition is an inherent property of the quadratic coefficients in the nonlinear terms of our PDE.  Formally, it states that the quadratic terms are linearly degenerate causing them to decay enough to prove global existence.  Without a null condition, we may have a genuine nonlinearity (see \cite{John3}) meaning a global existence proof would not be possible via hyperbolic energy estimates.  The way we use the condition can be seen in the energy estimates (\ref{enest17})-(\ref{enest21}).

In our work with the null condition it is convenient to define
\begin{align}
& \Pee{1}(\omega) u =
u - \rthreeip{u}{\omega} \omega  \notag \\
& \Pee{2}(\omega) u = \rthreeip{u}{\omega} \omega \notag
\end{align}
for $u\in \mathbf{R}^3 $.  $\Pee{1}(\omega)$ is the orthogonal projection onto $\textrm{span}\{\omega\}^{\perp}$ and $\Pee{2}(\omega)$ is the orthogonal projection onto $\textrm{span}\{\omega\}$.  Again, the argument $\omega$ will be implied in calculations.

The elasticity tensor $\hat{A}$ satisfies a null condition restricting the quadratic interaction of shear waves.  More precisely, for $B_{pjk} ^{lmn} = \frac{\partial \hat{A}^{lm}_{pj}}{\partial H^k_n}, \hat{A}^{lm}_{pj} \delta^n_k $ we have

\begin{align}
B_{pjk} ^{lmn} (I) \omega _l \omega _m \omega _n \xi _{(1)} ^p \xi _{(2)} ^j \xi _{(3)} ^k = 0 \label{nc}
\end{align}
for $\omega \in S^2$, $\xi _{(i)} \in \textrm{span}~\{\omega\}^\perp$, $i=1,2,3$.  Using our projections we can write
\begin{align}
B_{pjk} ^{lmn} (I) \omega _l \omega _m \omega _n (\Pee{1} \xi _{(1)}) ^p (\Pee{1}\xi _{(2)}) ^j (\Pee{1}\xi _{(3)}) ^k = 0 \label{ncproj}
\end{align}
for all $\omega \in S^2$, $\xi _{(i)} \in \rthree$.  The tensor $\hat{A}^{lm}_{pj} \delta^n_k$ clearly satisfies the null condition since
\begin{align}
\hat{A}_{pj} ^{lm}(I) &\delta ^n_k \omega _l \omega _m \omega _n (\Pee{1} \xi _{(1)}) ^p (\Pee{1}\xi _{(2)}) ^j (\Pee{1}\xi _{(3)}) ^k \label{ncAd} \\
&= \hat{A}_{pj} ^{lm} (I)\omega _l \omega _m (\Pee{1} \xi _{(1)}) ^p (\Pee{1}\xi _{(2)}) ^j \rthreeip{\omega}{(\Pee{1}\xi _{(3)})}
\\
&= 0. \notag
\end{align}
If we write $\frac{\partial \hat{A}^{lm}_{pj}}{\partial H^k_n}$ in terms of the elasticity tensor $A$ we have
\begin{align}
\frac{\partial }{\partial H^k_n} \hat{A}_{pj} ^{lm}(H) &= -\frac{\partial }{\partial H^k_n} [A^{PJ}_{LM} (F) F^p_P F^j_J F^L_l F^M_m]
\end{align}
If we differentiate term by term, the most difficult term comes when the derivative falls on $A$
\begin{align}
-\frac{\partial }{\partial H^k_n} [A^{PJ}_{LM} (F)] F^p_P F^j_J F^L_l F^M_m&= -\frac{\partial }{\partial F^K_N} A^{PJ}_{LM} (F) \frac{\partial }{\partial H^k_n} [F^K_N] F^p_P F^j_J F^L_l F^M_m
\\
&= \frac{\partial }{\partial F^K_N} A^{PJ}_{LM} (F) F_k^K F^n_N F^p_P F^j_J F^L_l F^M_m \notag
\end{align}
which is equal to
$$\frac{\partial }{\partial F^k_n} A^{pj}_{lm} (I)$$
when evaluated at the identity.  The proof that this tensor satisfies (\ref{nc}) can be found in \cite{Sid1}. When the derivative falls on $F^p_P$ we have
\begin{equation}
-A^{PJ}_{LM} (F) \frac{\partial }{\partial H^k_n} [F^p_P] F^j_J F^L_l F^M_m  = A^{PJ}_{LM} (F) F^p_k F^n_P F^j_J F^L_l F^M_m
\end{equation}
which equals
$$A^{nj}_{lm} (I) \delta^p_k$$ at the identity.  This satisfies (\ref{nc}) by (\ref{ncAd}).  A similar argument holds for differentiation of the other $F$ terms.

\vspace{.2in}

For bookkeeping purposes we define the weighted $L^2$ norms
\begin{equation}
\Chi{\sigma}{\theta} = \displaystyle{\sum _\projind \sum _j \sum _{|\alpha|+a \leq \sigma -1 \atop a\leq \theta}}
\norm{\langle \ev{\projind}t - r \rangle \sproj{\projind} \deej \stwiddleoops{a}{\alpha} \Udot} \label{chi}
\end{equation}

\begin{equation}
\Xi _{\sigma,\theta}(\Udot) = \displaystyle{\sum _\projind \sum _j \sum _{|\alpha|+a \leq \sigma -1 \atop a\leq \theta
}}
\norm{\gamma \langle \ev{\projind}t - r \rangle \sproj{\projind} \deej \stwiddleoops{a}{\alpha} \Udot}\label{xi}
\end{equation}
and
\begin{equation}
\Psi _{\sigma,\theta}(\Udot) = \displaystyle{\sum _{
|\alpha|+a \leq \sigma -1 \atop a\leq \theta
}}
\norm{\eta \ctwot  \grad \stwiddleoops{a}{\alpha} \Udot} \label{psi}
\end{equation}
These quantities will be used to link our local energy decay estimates to the energy estimates via our bootstrapping lemma to follow.

\section{Bound for the Pressure}

\setcounter{equation}{0}
\renewcommand{\theequation}{7.\arabic{equation}}

To eliminate the pressure term in local energy decay, we will bound it by the nonlinear terms $M^H$ and $N^v$ using the equations and constraints for $\Hdot$.

\vspace{.1in}

\begin{lemma} \label{pressure} Let $a=0,1,2,3...$, and $\alpha$ be a suitable exponent for $\Upsilon$.  Then for $p$ satisfying equations (\ref{hdoteqc2a})-(\ref{hdoteqc2b}) and constraints (\ref{hdotconst1})-(\ref{hdotconst3}),
$$\norm{\grad \stwiddleoops{a}{\alpha} p} \lesssim \norm{\soops{a}{\alpha}N^v (\Udot) } + \norm{ \soops{a}{\alpha} M^H(\Hdot)}.$$
\end{lemma}

\begin{proof} If we apply derivatives $\soops{a}{\alpha}$ to equation (\ref{hdoteqc2a})-(\ref{hdoteqc2b}) and rearrange we can write
\begin{align}
\grad \stwiddleoops{a}{\alpha} p &= \soops{a}{\alpha} \grad p \label{p1} \\
&= \soops{a}{\alpha} ( N^v(\Udot) +(1-c_1^2) M^H(\Hdot) -  \grad \cdot \Hdot +\visc \laplacian \vdot  -  \deet \vdot). \notag
\end{align}
We proceed by taking the divergence of the equation to obtain
\begin{align}
\laplacian \stwiddleoops{a}{\alpha} p = &  \grad  \cdot \soops{a}{\alpha} N^v(\Udot) +(1-c_1^2) \grad \cdot \soops{a}{\alpha} M^H(\Hdot)\label{p2} \\
&-  \grad \cdot (\stwiddleoops{a}{\alpha} \grad \cdot \Hdot) +\visc \grad \cdot \soops{a}{\alpha}  \laplacian \vdot  -  \grad \cdot \soops{a}{\alpha} \deet \vdot. \notag
\end{align}
Using the constraints (\ref{hdotconst1})-(\ref{hdotconst3}) we can simplify the last three terms:
\begin{align}
-  \grad \cdot (\stwiddleoops{a}{\alpha} \grad \cdot \Hdot) &= -  \deei \partial ^j \soops{a}{\alpha} \hdot{i}{j} \label{p3}\\
&=- \partial ^j \stwiddleoops{a}{\alpha} \deei \hdot{i}{j} \notag \\
&=- \partial ^j \stwiddleoops{a}{\alpha} \deej \tr \Hdot \notag \\
&=- \grad \cdot \stwiddleoops{a}{\alpha} M^H(\Hdot). \notag
\end{align}
and
\begin{align}
\visc \grad \cdot \soops{a}{\alpha} \laplacian \vdot  -  \grad \cdot \soops{a}{\alpha} \deet \vdot
= (\visc \stwiddleoops{a}{\alpha} \laplacian - \stwiddleoops{a}{\alpha} \deet)(\grad \cdot \vdot)=0. \label{p4}
\end{align}
So (\ref{p2}) simplifies to
\begin{align}
\laplacian \stwiddleoops{a}{\alpha} p = &  \grad  \cdot \soops{a}{\alpha} N^v(\Udot) -c_1^2 \grad \cdot \soops{a}{\alpha} M^H(\Hdot).\label{p5}
\end{align}
Using the notation $\Phi ^i =    \soops{a}{\alpha} (N^v)^i(\Udot) -c_1^2 \soops{a}{\alpha} (M^H)^i(\Hdot)$, we can estimate
\begin{align}
\norm{\grad \stwiddleoops{a}{\alpha} p}^2 &= -\innerprod{\laplacian \stwiddleoops{a}{\alpha} p}{ \stwiddleoops{a}{\alpha} p} \label{p6} \\
&=-\innerprod{\grad \cdot \Phi }{ \stwiddleoops{a}{\alpha} p} \notag \\
&=\innerprod{\Phi}{ \grad \stwiddleoops{a}{\alpha} p} \notag \\
&\leq \norm{\Phi} \cdot \norm{\grad \stwiddleoops{a}{\alpha} p} \notag
\end{align}
which implies the result.  \end{proof}

\section{Sobolev Inequalities}

\setcounter{equation}{0}
\renewcommand{\theequation}{8.\arabic{equation}}

We need several Sobolev inequalities beyond the standard embedding $H^2(\mathbf{R}^3) \hookrightarrow L^\infty (\mathbf{R}^3)$ to do our bootstrapping and energy estimates.

\vspace{.1in}

\begin{theorem} \label{sobolev} For $\ev{\projind}\in \mathbf{R}$, $\lambda\in [0,2]$, $f\in C^\infty _0 (\mathbf{R}^3)$, $r=|x|$, and $\rho = |y|$, we have the following:
\begin{align}
r|f(x)| \lesssim &\displaystyle{\sum_{|\alpha |\leq 1} \norm{\deerho \omegatwiddle ^\alpha f(y)}^{1/2}_{L^2(|y|\geq r)}
\times \sum_{|\alpha |\leq 2} \norm{ \omegatwiddle ^\alpha f(y)}^{1/2}_{L^2(|y|\geq r)}} \label{sob1} \\
\notag \\
r\langle \ev{\projind}t - r \rangle ^{1/2}|f(x)| \lesssim &\displaystyle{\sum_{|\alpha |\leq 1} \norm{\langle \ev{\projind}t - \rho \rangle\deerho \omegatwiddle ^\alpha f(y)}^{1/2}_{L^2(|y|\geq r)}} \label{sob2} \\
&\times \sum_{|\alpha |\leq 2} \norm{ \omegatwiddle ^\alpha f(y)}^{1/2}_{L^2(|y|\geq r)} \notag \\
\notag \\
|f(x)| \lesssim & \displaystyle{\sum_{|\alpha |\leq 1} \norm{\rho ^{-\lambda} \deerho \omegatwiddle ^\alpha f(y)}^{1/2}_{L^2(|y|\geq r)}} \label{sob3} \\
&\times \sum_{|\alpha |\leq 2} \norm{ \rho ^{\lambda - 2}\omegatwiddle ^\alpha f(y)}^{1/2}_{L^2(|y|\geq r)}. \notag
\end{align}

\end{theorem}

\begin{proof}
Estimates (\ref{sob1}) and (\ref{sob2}) are proven in \cite{Sid1} Lemma 3.3.

To derive (\ref{sob3}) first
assume $g(x)\in C^\infty _0 (\mathbf{R}^3)$.  Then using Cauchy-Schwarz and calculus facts we estimate
\begin{align}
\int _{S^2} |g(r \omega)|^4 d\omega \lesssim & \int ^\infty _{r} \int_{S^2}  |\deerho g(\rho \omega ) | \cdot |g(\rho \omega )|^3 d\rho d\omega \label{sobproof1}
\\
\lesssim & \int _{\rho \geq r} \rho ^{-2} |\deerho g(y ) | \cdot |g(y )|^3 dy \notag
\\
\lesssim & \left ( \int _{\rho \geq r} \rho ^{-2\lambda} |\deerho g(y ) | ^2 dy \right )^{1/2}  \left ( \int _{\rho \geq r} \rho ^{2(\lambda -2)} |g(y )|^6 dy \right )^{1/2}. \notag
\end{align}
Using Gagliardo-Nirenberg, we have

\begin{equation}
\norm{g(r\omega)} _{L^6(S^2)} \lesssim \sum _{|\alpha |\leq 1}
\norm{\omegatwiddle ^\alpha g(r\omega)} _{L^2(S^2)} ^{1/3}
\norm{g(r\omega)} _{L^4(S^2)} ^{2/3} .
\label{sobproof2}
\end{equation}
so

\begin{align}
\int _{\rho \geq r} \rho ^{2(\lambda -2)} |g(y )|^6 dy =& \int ^\infty _{r} \int_{S^2} \rho ^{2(-2+\lambda)+2} |g(\rho \omega )|^6 d\omega d\rho \label{sobproof3}
\\
\lesssim & \int ^\infty _{r} \Big \{\rho ^{2(\lambda -2)+2}
\left ( \int _{S^2} |g(\rho \omega)|^4 d\omega \right ) \notag
\\
&\times \left ( \sum _{|\alpha |\leq 1} \int _{S^2} |\omegatwiddle ^\alpha g(\rho \omega)|^2 d\omega \right ) \Big \}d\rho \notag
\\
\lesssim & \sup _{\rho \geq r} \int _{S^2} |g(\rho \omega)|^4 d\omega \notag
\\
& \times \sum _{|\alpha |\leq 1} \int _{\rho \geq r} \rho ^{2(\lambda -2)} |\omegatwiddle ^\alpha g(y)|^2 dy. \notag
\end{align}
Putting (\ref{sobproof1}) together with (\ref{sobproof3}), we get

\begin{align}
\left ( \int _{S^2} |g(r \omega)|^4 d\omega \right )^{1/2} \lesssim & \left ( \int _{\rho \geq r} \rho ^{-2\lambda} |\deerho g(y ) | ^2 dy \right )^{1/2} \label{sobproof4}
\\
& \times \left (\sum _{|\alpha |\leq 1}  \int _{\rho \geq r} \rho ^{2(\lambda -2)} |\omegatwiddle ^\alpha g(y)|^2 dy \right ) ^{1/2}. \notag
\end{align}
Combining (\ref{sobproof4}) with the isoperimetric Sobolev inequality
\begin{equation}
|f(x)| \lesssim \sum _{|\alpha |\leq 1} \norm{\omegatwiddle ^\alpha f(r\omega)} _{L^4(S^2)}, \label{isosob}
\end{equation}
we complete the proof of (\ref{sob3}).  \end{proof}

\vspace{.2in}

\begin{lemma} \label{hardy} (Hardy) For $f\in C^\infty _0 (\mathbf{R}^3)$, $r=|x|$, and $\rho = |y|$,

\begin{equation}
\norm{\rho ^{-1} f(y)} \lesssim \norm{\deerho f(y)}. \label{hardy1}
\end{equation}
\end{lemma}

\begin{proof}
The result is implied by the estimate
\begin{align}
\norm{\rho ^{-1} f(y)} ^2 =& \int \rho ^{-2} |f(y)|^2 dy \notag
\\
= &  \int_0^\infty \int _{S^2} |f(\rho \omega)|^2 d\omega d\rho \notag
\\
= &  \int_0^\infty \int _{S^2} \deerho (\rho) |f(\rho \omega)|^2 d\omega d\rho \notag
\\
\leq &  \int_0^\infty \int _{S^2} 2\rho |\deerho f(\rho \omega)|\cdot |f(\rho \omega)| d\omega d\rho \notag
\\
=& 2\int \rho ^{-1} |\deerho f(y)|\cdot |f(y)| dy  \notag
\\
\leq & 2 \norm{\deerho f(y)}\cdot \norm{\rho ^{-1} f(y)}. \notag
\end{align}  \end{proof}

\vspace{.2in}

\begin{corollary} \label{sobcor}
Let $\omega = \frac{x}{|x|}$, $r=|x|$, and  $\Udot=(\Hdot, \vdot)$ satisfy constraints (\ref{hdotconst1})-(\ref{hdotconst3}).  If
$$\E{a+|\alpha|+2}{a} [U],\Chi{a+|\alpha|+3}{a},\Psi_{a+|\alpha|+3,a}(\Udot)<\infty,$$
then

\begin{align}
\langle r \rangle |\soops{a}{\alpha} \Udot | \lesssim &
 \sum_{|\beta |\leq 2} \norm{ \soops{a}{\alpha +\beta} \Udot } \label{sob4} \\
\langle r \rangle \langle \ev{\projind}t - r \rangle ^{1/2} |\sproj{\projind} \soops{a}{\alpha} \Udot | \lesssim &
 \sum_{|\beta |\leq 2} \norm{ \soops{a}{\alpha +\beta} \Udot}  \label{sob5}\\
 &+ \Chi{a+|\alpha|+3}{a} \notag \\
 \etalinfnorm{ \ctwot \soops{a}{\alpha} \Udot } \lesssim &\sum_{|\beta |\leq 2} \norm{\soops{a}{\alpha + \beta} \Udot } \label{sob6} \\  &+ \Psi_{a+|\alpha|+3,a}(\Udot)\notag
\\
r^{3/2}| \omega \cdot (\soops{a}{\alpha}\Hdot \omega) | \lesssim & \sum_{|\beta |\leq 2} \norm{\soops{a}{\alpha + \beta} \Udot } \label{sob7} \\
&+ \sum_{|\beta |\leq 1} \norm{r \soops{a}{\alpha + \beta} M^H(\Hdot)} \notag \\
r^{3/2}| \omega \cdot (\soops{a}{\alpha}\vdot) | \lesssim & \sum_{|\beta |\leq 2} \norm{\soops{a}{\alpha + \beta} \Udot}, \label{sob8}
\\ \notag
\end{align}
where $\eta$ is defined as in (\ref{eta}), $E$ as in (\ref{energy}), $\mathcal{X}$ and $\Psi$ as in (\ref{chi}) and (\ref{psi}), and $M^H$ as in (\ref{hdoteqc2b}).
\end{corollary}

\begin{proof}
Inequalities (\ref{sob4}) and (\ref{sob5}) are proven in \cite{Sid1} Proposition 3.3.  To prove (\ref{sob6}) notice that on supp$~\eta$ we have
\begin{equation}
| \ctwot \soops{a}{\alpha} \Udot|
\leq  |\eta \ctwot \soops{a}{\alpha} \Udot| + \norm{\ctwot \soops{a}{\alpha} \Udot}_{L^{\infty}(0\leq \eta \leq 1)}. \label{sob9}
\end{equation}
Applying (\ref{sob3}) to the first term on the right with $\lambda =1$ and $f=\eta \ctwot \soops{a}{\alpha} \Udot$ gives
\begin{align}
|\eta \ctwot \soops{a}{\alpha} \Udot|
\lesssim & \displaystyle{\sum_{|\beta |\leq 1} \norm{\rho ^{-1} \deerho \omegatwiddle ^\beta (\eta \ctwot \soops{a}{\alpha} \Udot)}^{1/2}}  \\
&\times \sum_{|\beta |\leq 2} \norm{ \rho ^{-1}\omegatwiddle ^\beta (\eta \ctwot \soops{a}{\alpha} \Udot)}^{1/2}. \notag
\end{align}
Applying (\ref{hardy1}) to both terms on the right and using Young's inequality we have
\begin{align}
|\eta \ctwot \soops{a}{\alpha} \Udot| \lesssim & \displaystyle{\sum_{|\beta |\leq 1} \norm{ \deerho^2 \omegatwiddle ^\beta (\eta \ctwot \soops{a}{\alpha} \Udot)}} \label{sob10}
\\
&+ \sum_{|\beta |\leq 2} \norm{ \deerho \omegatwiddle ^\beta (\eta \ctwot \soops{a}{\alpha} \Udot)}. \notag
\end{align}
Using the fact that $\omegatwiddle \eta = \eta \omegatwiddle$ we can bound the first term on the right
\begin{align}
\sum_{|\beta |\leq 1} &\norm{ \deerho^2 \omegatwiddle ^\beta (\eta \ctwot \soops{a}{\alpha} \Udot)}
\\
\lesssim & \sum_{|\beta |\leq 1} (\norm{\deerho(\eta ' \soops{a}{\alpha+\beta} \Udot)} + \norm{\deerho(\eta \ctwot \deerho \soops{a}{\alpha+\beta} \Udot)}) \notag
\\
\lesssim & \sum_{|\beta |\leq 1} (\norm{\ctwot^{-1} \eta '' \soops{a}{\alpha+\beta} \Udot} + \norm{\eta '  \grad \soops{a}{\alpha+\beta} \Udot} \notag
\\
&+ \norm{\eta \ctwot \grad ^2 \soops{a}{\alpha+\beta} \Udot}) \notag
\\
\lesssim & \sum_{|\beta |\leq 2} \norm{\soops{a}{\alpha + \beta} \Udot }  + \Psi_{a+|\alpha|+3,a}(\Udot). \notag
\end{align}
The second term on the right of (\ref{sob10}) is bounded similarly.  To complete the proof of (\ref{sob6}) we use $\ctwot \lesssim \langle r \rangle$ on $\{ 0\leq \eta \leq 1 \}$ and (\ref{sob4}) on the second term on the right in (\ref{sob9}) to obtain
\begin{align}
\norm{\ctwot \soops{a}{\alpha} \Udot}_{L^{\infty}(0\leq \eta \leq 1)} \lesssim \norm{\langle r \rangle \soops{a}{\alpha} \Udot}_{L^{\infty}} \lesssim \sum_{|\beta |\leq 2} \norm{\soops{a}{\alpha + \beta} \Udot }.
\end{align}

Proving (\ref{sob7}) and (\ref{sob8}) begins with applying (\ref{sob3}) with $\lambda =\frac{1}{2}$ and $f=r^{3/2} \omega \cdot g$ to get
\begin{align}
r^{3/2} |\omega \cdot g| \lesssim & \displaystyle{\sum_{|\alpha |\leq 1} \norm{\rho ^{-1/2} \deerho \omegatwiddle ^\alpha (\rho^{3/2} \omega \cdot g) }^{1/2}_{L^2(|y|\geq r)}} \label{sob11}
\\
&\times \sum_{|\alpha |\leq 2} \norm{ \rho ^{-3/2}\omegatwiddle ^\alpha (\rho^{3/2} \omega \cdot g)}^{1/2}_{L^2(|y|\geq r)} \notag
\\
\lesssim & \displaystyle{\sum_{|\alpha |\leq 1} \norm{\rho ^{-1/2} \deerho (\rho^{3/2}  \omega \cdot \omegatwiddle ^\alpha g) }} \notag
\\
&+\sum_{|\alpha |\leq 2} \norm{ \omega \cdot \omegatwiddle ^\alpha   g} \notag
\\
\lesssim & \displaystyle{\sum_{|\alpha |\leq 1} \norm{\rho    \omega^i \deerho  \omegatwiddle ^\alpha g^i }} +\sum_{|\alpha |\leq 2} \norm{ \omegatwiddle ^\alpha   g}. \notag
\end{align}
Using the identity (\ref{polargrad}) we can estimate the first term on the right
\begin{align}
\displaystyle{\sum_{|\alpha |\leq 1} \norm{r    \omega^i \deer  \omegatwiddle ^\alpha g^i }} = & \displaystyle{\sum_{|\alpha |\leq 1} \norm{r    (\deei + r^{-1}(\Omega \wedge \omega)_i)  \omegatwiddle ^\alpha g^i }} \label{sob12}
\\
\lesssim & \displaystyle{\sum_{|\alpha |\leq 1} \norm{r \omegatwiddle ^\alpha \deei g^i }} + \sum_{|\alpha |\leq 2} \norm{ \omegatwiddle ^\alpha   g} \notag
\end{align}
In light of (\ref{sob11}), taking $g=\soops{a}{\alpha} \vdot$ yields (\ref{sob8}) after applying (\ref{hdotconst1}) while taking $g=\soops{a}{\alpha}H\omega$ and applying (\ref{hdotconst2}) and (\ref{hdotconst3}) gives
\begin{align}
\sum_{|\beta |\leq 1} \norm{r \omegatwiddle ^\beta \deei (\soops{a}{\alpha} \hdot{i}{j} \omega^j) } &\lesssim \sum_{|\beta |\leq 1} \big ( \norm{ \grad \omegatwiddle ^\beta \soops{a}{\alpha} \Hdot } + \norm{r \omegatwiddle ^\beta \deei \soops{a}{\alpha} \hdot{i}{j} \omega^j } \big )
\\
&\lesssim \sum_{|\beta |\leq 1} \big ( \norm{ \grad \omegatwiddle ^\beta \soops{a}{\alpha} \Hdot } + \norm{r \omegatwiddle ^\beta \stwiddleoops{a}{\alpha} M^H(\Hdot) } \big ), \notag
\end{align}
which implies (\ref{sob7}).   \end{proof}

We remark that the constraints (\ref{hdotconst1})-(\ref{hdotconst3}) are only necessary to prove (\ref{sob7}) and (\ref{sob8}).

\section{Bootstrapping and the Energy Identity}

\setcounter{equation}{0}
\renewcommand{\theequation}{9.\arabic{equation}}

Before we can begin bootstrapping we need a technical lemma to deal with terms where multiple derivatives fall on the elasticity tensor $\hat{A}$.  A similar lemma appeared in \cite{ST1}.

\begin{lemma} \label{ahatlemma}
Suppose $U\in H^{\sigma , \theta} _\Gamma$ with $\sigma \geq 3$.  Set $\sigma '=[\sigma /2]+2$.  Suppose $\E{\sigma '}{\sigma '} [U] <1$ and $|\Udot| \leq \delta$ for all $t\in\mathbf{R}^+$, with $\delta$ sufficiently small.  If $d$ is any positive integer then for $f:(\mathbf{R}^3 \otimes \mathbf{R}^3 )\times \mathbf{R}^3 \rightarrow \mathbf{R}^d$ satisfying $|f(\Udot)| = O(|\Udot|^p)$ at the origin we have the pointwise estimate

$$|\soops{b}{\beta} f(\Udot (t,x) )| \lesssim \sum_{{b_1+...+b_p \leq b}\atop{|\beta_1|+...+|\beta_p|\leq |\beta| }} |\soops{b_1}{\beta_1} \Udot (t,x)| \cdot \cdot \cdot |\soops{b_p}{\beta_p}\Udot (t,x)|$$
for $b+|\beta| \leq \sigma, |\beta| \leq \theta$.
\end{lemma}

\begin{proof}
Using the chain rule, write

\begin{equation} \label{chainf} \soops{b}{\beta} [f(\Udot)] = \sum_{{j\leq b}\atop{k\leq|\beta|}} \sum_{{b_1+...+b_{j+k} = b}\atop{\beta_1+...+\beta_{j+k}=\beta}} f^{(j+k)}(\Udot)\soops{b_1}{\beta_1} \Udot \cdot \cdot \cdot \soops{b_{j+k}}{\beta_{j+k}}\Udot \end{equation}
where $f^{(n)}$ denotes the $n$th derivative of $f$ with respect to $\Udot$. At most one derivative may exceed order $[\sigma /2]$ because $|\beta|+b \leq \sigma$.  By the standard Sobolev embedding, commutation properties and (\ref{normen}), we have

$$|\soops{c}{\varsigma} \Udot| \lesssim \norm{\grad^2 \soops{c}{\varsigma} \Udot} \lesssim \norm{(S+2)^c\grad^2 \Upsilon^\varsigma \Udot}  \lesssim  \E{\sigma '}{\sigma '} [U] \lesssim 1$$
whenever $c+|\varsigma| \leq [\sigma /2]$.  By the mean value theorem,

$$|f^{(j+k)}(\Udot)| \lesssim |\Udot|^{p-j-k}, ~j+k\leq p$$
for $|\Udot| \leq 1$.  Our result now follows from (\ref{chainf}).  \end{proof}

In order to make use of our local energy decay estimates we use our assumptions of the smallness of the initial data and the hyperbolic nature of the system to bound our weighted norms $\Xi _{\sigma,\theta}$,$\Psi _{\sigma,\theta}$, and $\mathcal{X}_{\sigma,\theta}$ defined in (\ref{chi})-(\ref{psi})  by the energy.  We accomplish this by bootstrapping the nonlinearity.

\begin{lemma}
\label{bootstrap}
Fix $\kappa\geq 15$ and $\mu=\kappa-4$.  Suppose $\kappa - 1 \geq \sigma \geq 1$, $\mu \geq \theta \geq 1$ and $\Udot (t,x)$ solves equations (\ref{hdoteqc2a})-(\ref{hdoteqc2b}) and satisfies constraints (\ref{hdotconst1})-(\ref{hdotconst3}).  As long as $\E{\mu}{\mu}^{1/2}[U]$ is sufficiently small
\begin{align}
\Xi _{\sigma,\theta}(\Udot) \lesssim \E{\sigma+1}{\theta+1}^{1/2}[U] \label{bootlem1} \\
\notag \\
\Psi _{\sigma,\theta}(\Udot) \lesssim \E{\sigma+1}{\theta+1}^{1/2}[U] \label{bootlem2} \\
\notag \\
\Chi{\sigma}{\theta} \lesssim \E{\sigma+1}{\theta+1}^{1/2}[U]  \label{bootlem3} \\ \notag
\end{align}
each hold whenever the quantities on the right make sense.
\end{lemma}

\begin{proof} Combining Theorem \ref{ledthm}, Lemma \ref{pressure} and equations (\ref{hdoteqc2a})-(\ref{hdoteqc2b}) gives us
\begin{align}
\Xi _{\sigma,\theta}(\Udot) = & \!
\sum_\projind \sum _j \sum _{|\alpha|+a \leq \sigma -1 \atop a\leq \theta}  \{ \norm {\gamma \langle \ev{\projind}t - r \rangle \sproj{\projind} \deej \stwiddleoops{a}{\alpha} \Udot }  \label{boot1} \\
 \lesssim & \sum _{|\alpha|+a \leq \sigma -1 \atop a\leq \theta}
\{ \norm{ \omegatwiddle \stwiddleoops {a}{\alpha} \Udot}
+ \norm{ \stwiddleoops {a+1}{\alpha} \Udot}
+ \norm{\grad \stwiddleoops {a}{\alpha} \vdot}  \notag \\
&  + \norm{ \stwiddleoops {a}{\alpha} \vdot}
+ \norm{ \gamma t \soops {a}{\alpha} N^H (\Hdot, \vdot)} \notag \\
& + \norm{ \gamma t \soops {a}{\alpha} (N^v (\Hdot, \vdot)-(c_1^2-1)M^H(\Hdot) - \grad p)} \}  \notag \\
\lesssim & \sum _{|\alpha|+a \leq \sigma \atop a\leq \theta+1} \norm{\stwiddleoops{a}{\alpha} \Udot} \notag \\
& +\sum _{|\alpha|+a \leq \sigma -1 \atop a\leq \theta} t \{\norm{\soops{a}{\alpha} N(\Udot)} +\norm{\soops{a}{\alpha} M^H(\Hdot)} \} \notag
\end{align}
and
\begin{align}
\Psi _{\sigma,\theta}(\Udot) =& \displaystyle{\sum _{
|\alpha|+a \leq \sigma -1 \atop a\leq \theta
}}
\norm{\eta \ctwot  \grad \stwiddleoops{a}{\alpha} \Udot} \label{boot2} \\
\lesssim & \sum _{|\alpha|+a \leq \sigma -1 \atop a\leq \theta} \{ \norm{ \grad \stwiddleoops {a+1}{\alpha} \Hdot}
+ \norm{ \stwiddleoops {a+1}{\alpha} \Udot}
+ \norm{\grad \stwiddleoops {a}{\alpha} \Udot}  \notag \\
& + \norm{ \stwiddleoops {a}{\alpha} \Udot}
+ \norm{ \eta t \soops {a}{\alpha} N^H (\Hdot, \vdot)} \notag \\
&+ \norm{ \eta t \soops {a}{\alpha} (N^v (\Hdot, \vdot)-(c_1^2-1)M^H(\Hdot) - \grad p)} \notag \\
&+ \norm{ \eta t \grad \cdot \soops {a}{\alpha} N^H (\Hdot, \vdot)} \}
\notag \\
\lesssim & \sum _{|\alpha|+a \leq \sigma +1 \atop a\leq \theta+1} \norm{\stwiddleoops{a}{\alpha} \Udot}
+ \sum _{|\alpha|+a \leq \sigma -1 \atop a\leq \theta} \norm{ \eta t \grad \cdot \soops {a}{\alpha} N^H (\Hdot, \vdot)} \notag \\
& +\sum _{|\alpha|+a \leq \sigma -1 \atop a\leq \theta} t \{\norm{\soops{a}{\alpha} N(\Udot)} +\norm{\soops{a}{\alpha} M^H(\Hdot)} \} \notag \\
& \notag
\end{align}
Our strategy will be first to bound $\Psi_{\sigma,\theta}$ by $\E{\sigma+1}{\theta+1}^{1/2}$ for low energy levels ($\sigma \leq \mu-1$), then use this result to bound $\Xi_{\sigma,\theta}$ by $\E{\sigma+1}{\theta+1}^{1/2}$ for low energy levels.  We will use these results to bound $\Psi_{\sigma,\theta}$ and $\Xi_{\sigma,\theta}$ by $\E{\sigma+1}{\theta+1}^{1/2}$ at high energy levels ($\mu \leq \sigma \leq \kappa-1$).

Assuming $\sigma \leq \mu -1$ we begin by estimating using Lemma \ref{ahatlemma}, splitting into our interior and exterior cutoffs and applying (\ref{sob4}) to obtain
\begin{align}
\sum _{|\alpha|+a \leq \sigma -1 \atop a\leq \theta}
t & \norm{\soops{a}{\alpha} N(\Udot)} \label{boot3}
\\
\lesssim &\sum _{|\alpha|+a \leq \sigma -1 \atop a\leq \theta}
\sum _{c+d\leq a \atop  |\varsigma| + |\vartheta| \leq|\alpha| }
\{ \norm{ \eta t  |\soops{c}{\varsigma} \Udot| \cdot |\grad \stwiddleoops{d}{\vartheta} \Udot|} \notag
\\
& + \norm{ \gamma t  |\soops{c}{\varsigma} \Udot| \cdot |\grad \stwiddleoops{d}{\vartheta} \Udot|} \} \notag
\\
\lesssim & \sum _{|\alpha|+a \leq \sigma -1 \atop a\leq \theta}  \sum _{c+d\leq a \atop  |\varsigma| + |\vartheta| \leq |\alpha| }
\{ \norm{ \eta   \ctwot | \soops{c}{\varsigma} \Udot| \cdot |\grad \stwiddleoops{d}{\vartheta} \Udot|} \notag
\\
& + \norm{ \langle r \rangle  |\soops{c}{\varsigma} \Udot| \cdot |\grad \stwiddleoops{d}{\vartheta} \Udot|} \} \notag
\\
\lesssim & \sum _{|\alpha|+a \leq \sigma -1 \atop a\leq \theta}  \sum _{c+d\leq a \atop  |\varsigma| + |\vartheta| \leq |\alpha| }
\{ \linfnorm{ \soops{c}{\varsigma} \Udot}\norm{ \eta   \ctwot  \grad \stwiddleoops{d}{\vartheta} \Udot} \notag
\\
& + \linfnorm{\langle r \rangle  \soops{c}{\varsigma} \Udot} \norm{  \grad \stwiddleoops{d}{\vartheta} \Udot} \} \notag
\\
\lesssim & \sum _{|\alpha|+a \leq \sigma -1 \atop a\leq \theta} \E{\mu}{\theta}^{1/2}[U] \notag
\{\norm{ \eta   \ctwot  \grad \stwiddleoops{a}{\alpha} \Udot} +\norm{  \grad \stwiddleoops{a}{\alpha} \Udot} \}. \notag
\end{align}
As long as $\E{\mu}{\mu}^{1/2}[U]\ll 1$ is small enough, this implies
\begin{align}
\sum _{|\alpha|+a \leq \sigma -1 \atop a\leq \theta}
t & \norm{\soops{a}{\alpha} N(\Udot)} \leq C \E{\sigma}{\theta}^{1/2}[U] + \epsilon \Psi_{\sigma,\theta}(\Udot) \label{boot4}
\end{align}
where $\epsilon \ll 1$. By a similar argument we have
\begin{align}
\sum _{|\alpha|+a \leq \sigma -1 \atop a\leq \theta}
t & \norm{\soops{a}{\alpha} M^H(\Hdot)} \leq C \E{\sigma}{\theta}^{1/2}[U] + \epsilon \Psi_{\sigma,\theta}(\Udot). \label{boot5}
\end{align}
The remaining term from (\ref{boot2}) is
\begin{align}
\sum _{|\alpha|+a \leq \sigma -1 \atop a\leq \theta} \norm{ \eta t \grad \cdot & \soops {a}{\alpha} N^H (\Hdot, \vdot)} \label{boot6}
\\
\lesssim &  \sum _{|\alpha|+a \leq \sigma -1 \atop a\leq \theta}  \sum _{c+d= a \atop  \varsigma + \vartheta =|\alpha| } \{ \norm{ \eta   \ctwot | \grad \soops{c}{\varsigma} \Udot| \cdot |\grad \stwiddleoops{d}{\vartheta} \Udot|} \notag
\\
& +  \norm{ \eta   \ctwot | \soops{c}{\varsigma} \Udot| \cdot |\grad ^2 \stwiddleoops{d}{\vartheta} \Udot|}\}. \notag
\end{align}
If we temporarily assume WLOG that $c+|\varsigma| \leq d+|\vartheta|$, then since $c+|\varsigma| + 3 < \frac{\mu}{2} +3 \leq \mu$ the first sum on the right can be estimated as follows

\begin{align}
\sum _{|\alpha|+a \leq \sigma -1 \atop a\leq \theta}  \sum _{c+d= a \atop  \varsigma + \vartheta =|\alpha| } & \norm{ \eta   \ctwot | \grad \soops{c}{\varsigma} \Udot| \cdot |\grad \stwiddleoops{d}{\vartheta} \Udot|} \label{boot7}
\\
\lesssim & \sum _{|\alpha|+a \leq \sigma -1 \atop a\leq \theta}  \sum _{c+d= a \atop  \varsigma + \vartheta =|\alpha| }
\linfnorm{\grad \soops{c}{\varsigma} \Udot} \norm{\eta   \ctwot \grad \stwiddleoops{d}{\vartheta} \Udot} \notag
\\
\lesssim & \sum _{|\alpha|+a \leq \sigma -1 \atop a\leq \theta} \E{\mu}{\theta}^{1/2}[U] \norm{\eta   \ctwot \grad \stwiddleoops{a}{\alpha} \Udot}  \notag \\
\lesssim & ~\epsilon \Psi_{\sigma,\theta}(\Udot) \notag
\end{align}
where again $\epsilon \ll 1$.  We will need to work harder to bound the second sum in (\ref{boot6}).  The exceptional case is when $d+|\vartheta|=a+|\alpha|=\sigma-1$ because there are too many derivatives to take the weight $\eta \ctwot$ with the $\grad^2$ term.  Instead we use (\ref{sob6}) to get the estimate

\begin{align}
\norm{ \eta   \ctwot | \Udot| & \cdot |\grad ^2 \stwiddleoops{a}{\alpha} \Udot|} \label{boot8}
\\
\lesssim & \linfnorm{\eta   \ctwot  \Udot} \norm{\grad ^2 \stwiddleoops{a}{\alpha} \Udot} \notag
\\
\lesssim & \left( \sum _{|\beta| \leq 2} \norm{\Upsilon ^\beta \Udot} + \Psi _{3,0}(\Udot)\right) \norm{\grad ^2 \stwiddleoops{a}{\alpha} \Udot} \notag
\\
\lesssim & ~C \E{\sigma+1}{\theta}^{1/2}[U]+\epsilon \Psi_{\sigma,\theta}(\Udot). \notag
\end{align}
When we are not in the exceptional case we have $d+|\vartheta|\leq \sigma -2$ so that
\begin{align}
\norm{ \eta   \ctwot | \soops{c}{\varsigma} \Udot| & \cdot |\grad ^2 \stwiddleoops{d}{\vartheta} \Udot|}  \label{boot9}
\\
\lesssim & \linfnorm{ \soops{c}{\varsigma} \Udot} \norm{ \eta   \ctwot \grad ^2 \stwiddleoops{d}{\vartheta} \Udot} \notag
\\
\lesssim & \epsilon \Psi_{\sigma,\theta}(\Udot) .\notag
\end{align}
This leaves us with the bound
\begin{align}
\sum _{|\alpha|+a \leq \sigma -1 \atop a\leq \theta} \norm{ \eta t \grad \cdot & \soops {a}{\alpha} N^H (\Hdot, \vdot)} \leq C \E{\sigma+1}{\theta}^{1/2}[U]+\epsilon \Psi_{\sigma,\theta}(\Udot) \label{boot10}
\end{align}
Altogether, (\ref{boot2}),  (\ref{boot4}),  (\ref{boot5}) , and (\ref{boot10}) give the estimate (\ref{bootlem2}) for $\sigma \leq \mu -1$.

Keeping the assumption $\sigma \leq \mu -1$ we have from (\ref{bootlem2}), (\ref{boot1}), (\ref{boot4}), and (\ref{boot5})
\begin{align}
\Xi _{\sigma,\theta}(\Udot)
\lesssim & \sum _{|\alpha|+a \leq \sigma \atop a\leq \theta+1} \norm{\stwiddleoops{a}{\alpha} \Udot} \label{boot11}
\\
& +\sum _{|\alpha|+a \leq \sigma -1 \atop a\leq \theta} t \{\norm{\soops{a}{\alpha} N(\Udot)} +\norm{\soops{a}{\alpha} M^H(\Hdot)} \} \notag
\\
\lesssim & ~\E{\sigma}{\theta+1}^{1/2}[U] + \Psi_{\sigma,\theta}(\Udot) \notag
\\
\lesssim & ~\E{\sigma +1}{\theta +1}^{1/2}[U] \notag
\end{align}
proving (\ref{bootlem1}) for $\sigma \leq \mu -1$.

For the rest of the proof we will assume $\mu \leq \sigma \leq \kappa-1$.  As in (\ref{boot3}) we have
\begin{align}
\sum _{|\alpha|+a \leq \sigma -1 \atop a\leq \theta}
t & \norm{\soops{a}{\alpha} N(\Udot)} \label{boot12}
\\
\lesssim &\sum _{|\alpha|+a \leq \sigma -1 \atop a\leq \theta}
\sum _{c+d\leq a \atop  |\varsigma| + |\vartheta| \leq|\alpha| }
\{ \norm{ \eta t  |\soops{c}{\varsigma} \Udot| \cdot |\grad \stwiddleoops{d}{\vartheta} \Udot|} \notag
\\
& + \norm{ \gamma t  |\soops{c}{\varsigma} \Udot| \cdot |\grad \stwiddleoops{d}{\vartheta} \Udot|} \} \notag
\\
\lesssim &\sum _{|\alpha|+a \leq \sigma -1 \atop a\leq \theta}
\sum _{c+d\leq a \atop  |\varsigma| + |\vartheta| \leq|\alpha| }
\{ \norm{ \eta t  |\soops{c}{\varsigma} \Udot| \cdot |\grad \stwiddleoops{d}{\vartheta} \Udot|} \notag
\\
& + \norm{ \gamma t  |\soops{c}{\varsigma} \Udot| \cdot |\grad \stwiddleoops{d}{\vartheta} \Udot|} \} . \notag
\end{align}
For terms where $d+|\vartheta| \leq \mu - 3$ we have
\begin{align}
\norm{ \eta t  &|\soops{c}{\varsigma} \Udot| \cdot |\grad \stwiddleoops{d}{\vartheta} \Udot|}
+ \norm{ \gamma t  |\soops{c}{\varsigma} \Udot| \cdot |\grad \stwiddleoops{d}{\vartheta} \Udot|} \label{boot13}
\\
\lesssim & \linfnorm{\soops{c}{\varsigma} \Udot} \norm{\eta t  \grad \stwiddleoops{d}{\vartheta} \Udot}
+ \linfnorm{\langle r \rangle \grad \stwiddleoops{d}{\vartheta} \Udot} \norm{ \soops{c}{\varsigma} \Udot} \notag
\\
\lesssim & ~\E{\sigma +1}{\theta +1}^{1/2}[U] \Psi_{\mu-1,\theta}(\Udot) + \E{\mu}{\theta}^{1/2}[U] \E{\sigma-1}{\theta}^{1/2}[U] \notag
\\
\lesssim & ~\E{\sigma +1}{\theta +1}^{1/2}[U], \notag
\end{align}
where we used $\Psi_{\mu-1,\theta}(\Udot)\lesssim \E{\mu}{\mu} ^{1/2} (\Udot) \lesssim 1$ in the final inequality.  On the other hand, for terms where $d+|\vartheta| \geq \mu - 2$, we know $c+|\varsigma| \leq 4$ allowing us to get the following bound:
\begin{align}
\norm{ \eta t  &|\soops{c}{\varsigma} \Udot| \cdot |\grad \stwiddleoops{d}{\vartheta} \Udot|}
+ \norm{ \gamma t  |\soops{c}{\varsigma} \Udot| \cdot |\grad \stwiddleoops{d}{\vartheta} \Udot|} \label{boot14}
\\
\lesssim & \linfnorm{\soops{c}{\varsigma} \Udot} \norm{\eta t  \grad \stwiddleoops{d}{\vartheta} \Udot}
+ \linfnorm{\langle r \rangle  \soops{c}{\varsigma} \Udot} \norm{ \grad\stwiddleoops{d}{\vartheta} \Udot} \notag
\\
\lesssim & \E{6}{\theta}^{1/2}[U] (\Psi_{\sigma,\theta}(\Udot)
+  \E{\sigma}{\theta}^{1/2}[U]) \notag
\\
\leq & C\E{\sigma}{\theta}^{1/2}[U] + \epsilon \Psi_{\sigma,\theta}(\Udot). \notag
\end{align}
Combining (\ref{boot12}), (\ref{boot13}) and (\ref{boot14}) we have the estimate
\begin{align}
\sum _{|\alpha|+a \leq \sigma -1 \atop a\leq \theta}
t & \norm{\soops{a}{\alpha} N(\Udot)}
\leq   C\E{\sigma+1}{\theta+1}^{1/2}[U] + \epsilon \Psi_{\sigma,\theta}(\Udot). \label{boot15}
\end{align}
And by a similar argument
\begin{align}
\sum _{|\alpha|+a \leq \sigma -1 \atop a\leq \theta}
t & \norm{\soops{a}{\alpha} M^H(\Hdot)}
\lesssim   \E{\sigma+1}{\theta+1}^{1/2}[U] + \epsilon \Psi_{\sigma,\theta}(\Udot). \label{boot16}
\end{align}
To finish the proof of (\ref{bootlem2}) we need to estimate the remaining nonlinear term in (\ref{boot2}).  Starting with (\ref{boot6}) we bound the first sum by assuming WLOG that $c+|\varsigma| \leq d+|\vartheta|$ because both terms contain a gradient. Then since $c+|\varsigma| + 3 < \frac{\kappa}{2} +3 \leq \mu$ we can follow estimate (\ref{boot7}).  As for the second sum on the right in (\ref{boot6}) we again use (\ref{sob6}) combined with (\ref{bootlem2})  in the exceptional case when $d+|\vartheta|=a+|\alpha|=\sigma-1$:
\begin{align}
\norm{ \eta   \ctwot | \Udot| & \cdot |\grad ^2 \stwiddleoops{a}{\alpha} \Udot|} \label{boot17}
\\
\lesssim & \linfnorm{\eta   \ctwot  \Udot} \norm{\grad ^2 \stwiddleoops{a}{\alpha} \Udot} \notag
\\
\lesssim & \left( \sum _{|\beta| \leq 2} \norm{\Upsilon ^\beta \Udot} + \Psi _{3,0}(\Udot)\right) \norm{\grad ^2 \stwiddleoops{a}{\alpha} \Udot} \notag
\\
\lesssim & ~ \left( \E{2}{0}^{1/2}[U]+ \E{4}{1}^{1/2}[U] \right) \E{\sigma+1}{\theta}^{1/2}[U] \notag
\\
\lesssim & ~ \E{\sigma+1}{\theta}^{1/2}[U]. \notag
\end{align}
When not in the exceptional case we have $d+|\vartheta|\leq \sigma -2$ which means if $c+|\varsigma| \leq \mu-2$ then
\begin{align}
\norm{ \eta   \ctwot | \soops{c}{\varsigma} \Udot| & \cdot |\grad ^2 \stwiddleoops{d}{\vartheta} \Udot|}  \label{boot18}
\\
\lesssim & \linfnorm{ \soops{c}{\varsigma} \Udot} \norm{ \eta   \ctwot \grad ^2 \stwiddleoops{d}{\vartheta} \Udot} \notag
\\
\lesssim & \E{\mu}{mu}^{1/2}[U] \Psi_{\sigma,\theta}(\Udot) \notag
\\
\lesssim & \epsilon \Psi_{\sigma,\theta}(\Udot) \notag
\end{align}
and if $c+|\varsigma| \geq \mu-1$ then $d+|\vartheta|\leq 4$ giving us
\begin{align}
\norm{ \eta   \ctwot | \soops{c}{\varsigma} \Udot| & \cdot |\grad ^2 \stwiddleoops{d}{\vartheta} \Udot|}  \label{boot19}
\\
\lesssim & \linfnorm{ \soops{c}{\varsigma} \Udot} \norm{ \eta   \ctwot \grad ^2 \stwiddleoops{d}{\vartheta} \Udot} \notag
\\
\lesssim & \E{\sigma+1}{\theta}^{1/2}[U] \Psi_{6,\theta}(\Udot) \notag
\\
\lesssim & \E{\sigma+1}{\theta}^{1/2}[U] \E{7}{\theta+1}^{1/2}[U] \notag
\\
\lesssim & \E{\sigma+1}{\theta}^{1/2}[U], \notag
\end{align}
after again applying (\ref{bootlem2}) for low energy.

Taking (\ref{boot2}), (\ref{boot7}),(\ref{boot15})-(\ref{boot19}) together we derive estimate (\ref{bootlem2}).  In light of (\ref{boot11}) we have also proven (\ref{bootlem1}).  The bound (\ref{bootlem3}) follows immediately since

$$
\Chi{\sigma}{\theta} \lesssim \Xi _{\sigma,\theta}(\Udot) +\Psi _{\sigma,\theta}(\Udot).
$$

\end{proof}

Before beginning our energy estimates we state the energy identity which can be obtained by a standard calculation for quasilinear symmetric hyperbolic systems.  Notice we have not commuted the vector fields with the laplacian in the viscosity term.  We will leave this commutation until the end of the energy estimates where we do induction on $\theta$.

\begin{align}
\deet & \E{\sigma}{\theta} [U] - \visc \sum _{{a+|\alpha| \leq \sigma}\atop {a\leq \theta}} \int \stwiddleoops{a}{\alpha} v^i  \soops{a}{\alpha} \laplacian v^i \label{compen19}
\\
= & \sum _{{a+|\alpha| \leq \sigma}\atop {a\leq \theta}}
\Big( \int [D^k_n \ahat{lm}{pj}](I) \big \{ \frac{1}{2} \deen v^k\stwiddleoops{a}{\alpha} \hdot{j}{m}  \stwiddleoops{a}{\alpha}\hdot{p}{l} \notag
\\
&+\deel \hdot{k}{n} \stwiddleoops{a}{\alpha} \hdot{j}{m}  \stwiddleoops{a}{\alpha} v^i \notag
\\
&- \sum _{{b+c=a}\atop {\beta+\varsigma = \alpha}} \sum _{c+|\varsigma| \neq a+|\alpha|} C_{b,a,\beta,\alpha} \soops{b}{\beta} \hdot{k}{n} \stwiddleoops{a}{\alpha} v^p  \deel \stwiddleoops{c}{\varsigma}\hdot{j}{m} \big \} \notag
\\
&+\int \ahat{lm}{pj}(H) \stwiddleoops{a}{\alpha} \hdot{j}{m}  \deel \hdot{p}{i} \stwiddleoops{a}{\alpha} v^i
\notag
\\
&+\sum _{{b+c=a}\atop {\beta+\varsigma = \alpha}} \sum _{c+|\varsigma| \neq a+|\alpha|}  \int C_{b,a,\beta,\alpha} \Big ( \stwiddleoops{a}{\alpha} v^i \soops{b}{\beta} v^p \deep \stwiddleoops{c}{\varsigma} v^i \notag
\\
& -\ahat{lm}{pj}(I) \big \{ \stwiddleoops{a}{\alpha} \hdot{j}{m} \soops{b}{\beta} \hdot{p}{i} \deel \stwiddleoops{c}{\varsigma} v^i
+ \stwiddleoops{a}{\alpha} v^i \soops{b}{\beta} \hdot{p}{i} \deel \stwiddleoops{c}{\varsigma}\hdot{j}{m} \big \} \notag
\\
&+\ahat{lm}{pj}(H) \stwiddleoops{a}{\alpha}\hdot{j}{m} \soops{b}{\beta} v^i \deei \stwiddleoops{c}{\varsigma} \hdot{p}{l} \Big )  \notag
\\
&+O \big(\sum _{{b+c+d=a}\atop {\beta+\varsigma+\vartheta = \alpha}} \sum _{c+|\varsigma| \neq a+|\alpha|} \int  |\stwiddleoops{a}{\alpha}\Udot|\cdot| \soops{b}{\beta}\Udot|\cdot| \soops{d}{\vartheta} \Udot|\cdot| \grad \stwiddleoops{c}{\varsigma}\Udot| \big ) \Big) \notag.
\end{align}
Here we have separated the quadratic terms which will be handled via the null condition and constraints from the higher order terms.  We also note that we have used Lemma \ref{ahatlemma} and the smallness of the low energy to get the higher order terms.

\section{Energy Estimates}

\setcounter{equation}{0}
\renewcommand{\theequation}{10.\arabic{equation}}

Now we prove the main theorem by estimating the time derivative of the energy $\deet \E{\sigma}{\theta}[U]$ for high ($\sigma=\kappa$) and low ($\sigma=\kappa-4=\mu$) energy levels.  Using these estimates we will do a finite induction proof on $\theta$ to complete the result.  We do not make any assumptions about $\theta$ until the inductive portion of the proof.
\begin{proof}[Proof of Theorem \ref{main}] {\ }  \newline

\noindent \textit{High Energy}

Beginning by assuming $\sigma = \kappa$ we estimate very roughly using (\ref{compen19}), Sobolev embedding, and crashing through with absolute value to obtain
\begin{align}
\deet & \E{\kappa}{\theta} [U] - \visc \sum _{{a+|\alpha| \leq \kappa}\atop {a\leq \theta}} \int \stwiddleoops{a}{\alpha} v^i  \soops{a}{\alpha} \laplacian v^i \label{enest1}
\\
\lesssim & \sum _{{a+|\alpha| \leq \kappa}\atop {a\leq \theta}} \sum _{{b+c=a}\atop {\beta+\varsigma = \alpha}} \sum _{c+|\varsigma| \neq a+|\alpha|}
\int  |\stwiddleoops{a}{\alpha}\Udot|\cdot| \soops{b}{\beta}\Udot|\cdot| \grad \stwiddleoops{c}{\varsigma}\Udot|.  \notag
\end{align}
We split our integral using the cutoff functions $\eta$ and $\gamma$
\begin{align}
\int  |\stwiddleoops{a}{\alpha}\Udot|&\cdot| \soops{b}{\beta}\Udot|\cdot| \grad \stwiddleoops{c}{\varsigma}\Udot| \label{enest2}
\\
=& \int \gamma |\stwiddleoops{a}{\alpha}\Udot|\cdot| \soops{b}{\beta}\Udot|\cdot| \grad \stwiddleoops{c}{\varsigma}\Udot| \notag
\\
&+\int \eta |\stwiddleoops{a}{\alpha}\Udot|\cdot| \soops{b}{\beta}\Udot|\cdot| \grad \stwiddleoops{c}{\varsigma}\Udot|. \notag
\end{align}

To estimate the exterior integral we consider two cases and use the fact that $r\geq \frac{\ctwot}{m}$ on supp $\gamma$.  If $|\beta| + b \leq |\varsigma|+c$, then $|\beta| + b \leq [\frac{\kappa}{2}] \leq \mu-2$, so
\begin{align}
\int \gamma |\stwiddleoops{a}{\alpha}\Udot|&\cdot| \soops{b}{\beta}\Udot|\cdot| \grad \stwiddleoops{c}{\varsigma}\Udot| \label{enest3}
\\
&\leq \int mr \ctwot^{-1} |\stwiddleoops{a}{\alpha}\Udot|\cdot| \soops{b}{\beta}\Udot|\cdot| \grad \stwiddleoops{c}{\varsigma}\Udot| \notag
\\
&\lesssim \langle t \rangle ^{-1} \int \langle r \rangle |\stwiddleoops{a}{\alpha}\Udot|\cdot| \soops{b}{\beta}\Udot|\cdot| \grad \stwiddleoops{c}{\varsigma}\Udot| \notag
\\
&\lesssim \langle t \rangle ^{-1} \linfnorm{ \langle r \rangle \soops{b}{\beta}\Udot}  \int  |\stwiddleoops{a}{\alpha}\Udot|\cdot| \grad \stwiddleoops{c}{\varsigma}\Udot| \notag
\\
&\lesssim \langle t \rangle ^{-1} \E{\mu}{\mu}^{1/2} [U] \E{\kappa}{\theta} [U]. \notag
\end{align}
Here we used the weighted Sobolev estimate (\ref{sob4}). If $|\beta| + b > |\varsigma|+c$, then $|\varsigma|+c \leq [\frac{\kappa}{2}]+1 \leq \mu-2$, so by a similar argument we have the same bound in this case.  Thus,
\begin{align}
\sum _{{a+|\alpha| \leq \kappa}\atop {a\leq \theta}}& \sum _{{b+c=a}\atop {\beta+\varsigma = \alpha}} \sum _{c+|\varsigma| \neq a+|\alpha|} \int \gamma |\stwiddleoops{a}{\alpha}\Udot|\cdot| \soops{b}{\beta}\Udot|\cdot| \grad \stwiddleoops{c}{\varsigma}\Udot| \label{enest4}
\\
&\lesssim \langle t \rangle ^{-1} \E{\mu}{\mu}^{1/2} [U] \E{\kappa}{\theta} [U]. \notag
\end{align}

On the interior region we also consider a few cases and this time we use (\ref{bootlem2}) to recover a factor of $\langle t \rangle ^{-1}$.  The exceptional case is when $|\varsigma|+c = |\alpha| + a -1$, and $|\beta| + b=1$ where we need to use (\ref{sob6}) in tandem with (\ref{bootlem2}).  In this case, we estimate
\begin{align}
\int \eta& |\stwiddleoops{a}{\alpha} \Udot|\cdot| \soops{b}{\beta}\Udot|\cdot| \grad \stwiddleoops{c}{\varsigma}\Udot| \label{enest5}
\\
\lesssim & \int \ctwot^{-1} | \stwiddleoops{a}{\alpha} \Udot|\cdot| \eta \ctwot \soops{b}{\beta}\Udot|\cdot| \grad \stwiddleoops{c}{\varsigma}\Udot| \notag
\\
\lesssim & \langle t \rangle ^{-1} \etalinfnorm{ \ctwot \soops{b}{\beta}\Udot} \int | \stwiddleoops{a}{\alpha} \Udot|\cdot| \grad \stwiddleoops{c}{\varsigma}\Udot| \notag
\\
\lesssim & \langle t \rangle ^{-1} \Big( \sum_{|\vartheta |\leq 2} \norm{\soops{b}{ \beta+\vartheta} \Udot } + \Psi_{b+|\beta|+3,b}(\Udot) \Big ) \E{\kappa}{\theta} [U] \notag
\\
\lesssim & \langle t \rangle ^{-1} \Big(\E{3}{1}^{1/2} [U] + \Psi_{4,1}(\Udot)\Big ) \E{\kappa}{\theta} [U] \notag
\\
\lesssim & \langle t \rangle ^{-1} \E{5}{2}^{1/2} [U] \E{\kappa}{\theta} [U] \notag
\\
\lesssim & \langle t \rangle ^{-1} \E{\mu}{\mu}^{1/2} [U] \E{\kappa}{\theta} [U]. \notag
\end{align}
A second possibility is $ |\beta| + b \leq |\varsigma|+c \leq |\alpha| + a -2$ in which case $|\beta| + b \leq [\frac{\kappa}{2}] \leq \mu-2$ giving the estimate
\begin{align}
\int \eta& |\stwiddleoops{a}{\alpha} \Udot|\cdot| \soops{b}{\beta}\Udot|\cdot| \grad \stwiddleoops{c}{\varsigma}\Udot| \label{enest6}
\\
\lesssim & \int \ctwot^{-1} | \stwiddleoops{a}{\alpha} \Udot|\cdot|  \soops{b}{\beta}\Udot|\cdot| \eta \ctwot \grad \stwiddleoops{c}{\varsigma}\Udot| \notag
\\
\lesssim & \langle t \rangle ^{-1} \linfnorm{  \soops{b}{\beta}\Udot} \norm{ \stwiddleoops{a}{\alpha} \Udot}\cdot \norm{ \eta \ctwot \grad \stwiddleoops{c}{\varsigma}\Udot} \notag
\\
\lesssim & \langle t \rangle ^{-1} \E{|\beta| + b+2}{b}^{1/2} [U] \E{\kappa}{\theta} ^{1/2} [U] \Psi_{|\varsigma|+c+1,c}(\Udot)  \notag
\\
\lesssim & \langle t \rangle ^{-1} \E{\mu}{\mu}^{1/2} [U] \E{\kappa}{\theta} ^{1/2} [U] \E{\kappa}{\theta+1} ^{1/2} [U]  \notag
\end{align}
using (\ref{bootlem2}) and the standard Sobolev embedding.
Finally, if $|\beta| + b > |\varsigma|+c$, then $|\varsigma|+c \leq [\frac{\kappa}{2}]-1 \leq \mu-5$, so
\begin{align}
\int \eta& |\stwiddleoops{a}{\alpha} \Udot|\cdot| \soops{b}{\beta}\Udot|\cdot| \grad \stwiddleoops{c}{\varsigma}\Udot| \label{enest7}
\\
\lesssim & \int \ctwot^{-1} | \stwiddleoops{a}{\alpha} \Udot|\cdot|  \soops{b}{\beta}\Udot|\cdot| \eta \ctwot \grad \stwiddleoops{c}{\varsigma}\Udot| \notag
\\
\lesssim &  \langle t \rangle ^{-1} \linfnorm{\eta \ctwot \grad \stwiddleoops{c}{\varsigma}\Udot} \norm{\stwiddleoops{a}{\alpha} \Udot} \cdot \norm{\soops{b}{\beta}\Udot} \notag
\\
\lesssim &  \langle t \rangle ^{-1} (\E{|\varsigma|+c+3}{c} ^{1/2} [U]+ \Psi_{|\varsigma|+c+4,c}(\Udot)) \E{\kappa}{\theta} [U] \notag
\\
\lesssim &  \langle t \rangle ^{-1} \E{|\varsigma|+c+5}{c+1} ^{1/2} [U] \E{\kappa}{\theta} [U] \notag
\\
\lesssim & \langle t \rangle ^{-1} \E{\mu}{\mu}^{1/2} [U] \E{\kappa}{\theta} [U] \notag
\end{align}
after using (\ref{sob6}) and (\ref{bootlem2}).
Taking estimates (\ref{enest5})-(\ref{enest7}) together, we have
\begin{align}
\sum _{{a+|\alpha| \leq \kappa}\atop {a\leq \theta}} &\sum _{{b+c=a}\atop {\beta+\varsigma = \alpha}} \sum _{c+|\varsigma| \neq a+|\alpha|} \int \eta |\stwiddleoops{a}{\alpha} \Udot|\cdot| \soops{b}{\beta}\Udot|\cdot| \grad \stwiddleoops{c}{\varsigma}\Udot| \label{enest9}
\\
\lesssim & \langle t \rangle ^{-1} \E{\mu}{\mu}^{1/2} [U] \E{\kappa}{\theta+1} [U]. \notag
\end{align}

Altogether our high energy estimates, (\ref{enest4}) and (\ref{enest9}), give
\begin{align}
\deet & \E{\kappa}{\theta} [U] - \visc \sum _{{a+|\alpha| \leq \kappa}\atop {a\leq \theta}} \int \stwiddleoops{a}{\alpha} v^i  \soops{a}{\alpha} \laplacian v^i \label{enest11}
\\
\lesssim & \sum _{{a+|\alpha| \leq \kappa}\atop {a\leq \theta}} \sum _{{b+c=a}\atop {\beta+\varsigma = \alpha}} \sum _{c+|\varsigma| \neq a+|\alpha|} \int |\stwiddleoops{a}{\alpha} \Udot|\cdot| \soops{b}{\beta}\Udot|\cdot| \grad \stwiddleoops{c}{\varsigma}\Udot| \notag
\\
\lesssim & \langle t \rangle ^{-1} \E{\mu}{\mu}^{1/2} [U] \E{\kappa}{\theta+1} [U]. \notag
\end{align}
In order to complete the induction on $\theta$ later on, we will also need the estimate
\begin{align}
\deet & \E{\kappa}{\mu+1} [U] - \visc \sum _{{a+|\alpha| \leq \kappa}\atop {a\leq \theta}} \int \stwiddleoops{a}{\alpha} v^i  \soops{a}{\alpha} \laplacian v^i \label{enest11a}
\\
\lesssim & \langle t \rangle ^{-1} \E{\mu}{\mu}^{1/2} [U] \E{\kappa}{\mu+1} [U]. \notag
\end{align}
Considering (\ref{enest4}) with $\theta = \mu+1$, we only need to show
\begin{align}
\sum _{{a+|\alpha| \leq \kappa}\atop {a\leq \mu+1}} &\sum _{{b+c=a}\atop {\beta+\varsigma = \alpha}} \sum _{c+|\varsigma| \neq a+|\alpha|} \int \eta |\stwiddleoops{a}{\alpha} \Udot|\cdot| \soops{b}{\beta}\Udot|\cdot| \grad \stwiddleoops{c}{\varsigma}\Udot| \label{enest11b}
\\
\lesssim & \langle t \rangle ^{-1} \E{\mu}{\mu}^{1/2} [U] \E{\kappa}{\mu+1} [U]. \notag
\end{align}
In terms where $a\leq \mu$, we may apply estimate (\ref{enest9}) to achieve this bound, so we only consider the case when $b+c=a=\mu+1$.  Terms where $b$ and $c$ are each at most $\mu$ are estimated using (\ref{enest5}) - (\ref{enest7}).  We are left with two exceptional cases, $b=\mu+1$ and $c=\mu+1$.  If $b=\mu+1$, then $c=0$ and $|\varsigma| \leq 3$, so using (\ref{sob6}) and (\ref{bootlem2}) we get
\begin{align}
\int \eta & |\stwiddleoops{a}{\alpha} \Udot|\cdot| \soops{b}{\beta}\Udot|\cdot| \grad \stwiddleoops{c}{\varsigma}\Udot| \label{enest11c}
\\
=& \int \eta  |\stwiddleoops{\mu+1}{\alpha} \Udot|\cdot| \soops{\mu+1}{\beta}\Udot|\cdot| \grad \Upsilon^{\varsigma}\Udot| \notag
\\
\lesssim & \langle t \rangle ^{-1}   \norm{\stwiddleoops{\mu+1}{\alpha} \Udot}\cdot \norm{ \soops{\mu+1}{\beta}\Udot} \cdot \linfnorm{  \eta \ctwot \grad \Upsilon^{\varsigma}\Udot} \notag
\\
\lesssim & \langle t \rangle ^{-1} \E{\kappa}{\mu+1} [U] (\E{6}{0}^{1/2} [U] + \Psi _{7,0} (\Udot)) \notag
\\
\lesssim & \langle t \rangle ^{-1} \E{\kappa}{\mu+1} [U] \E{8}{1}^{1/2} [U]  \notag
\\
\lesssim & \langle t \rangle ^{-1} \E{\kappa}{\mu+1} [U] \E{\mu}{\mu}^{1/2} [U].  \notag
\end{align}
A similar argument can be used when $c=\mu+1$.  This shows (\ref{enest11b}) holds, so the estimate (\ref{enest11a}) is valid as well.

\vspace{.1in}

\noindent \textit{Low Energy}

For the low energy ($\sigma= \mu$) we need to rely on the precise structure of (\ref{compen19}) combined with the null condition (\ref{ncproj}) to obtain $\langle t \rangle ^{-3/2}$ decay on the exterior region.  We split each integral on the RHS of (\ref{compen19}) into two integrals using our cutoffs $\eta$ and $\gamma$.  On the interior region, we will easily obtain the desired decay by estimating roughly as in (\ref{enest1})
\begin{align}
\int  \eta & |\stwiddleoops{a}{\alpha}\Udot|\cdot| \soops{b}{\beta}\Udot|\cdot| \grad \stwiddleoops{c}{\varsigma}\Udot| \label{enest12}
\\
\lesssim & \langle t \rangle ^{-2} \int  \ctwot ^{2} \eta |\stwiddleoops{a}{\alpha}\Udot|\cdot| \soops{b}{\beta}\Udot|\cdot| \grad \stwiddleoops{c}{\varsigma}\Udot| \notag
\\
\lesssim & \langle t \rangle ^{-2} \int    |\stwiddleoops{a}{\alpha}\Udot|\cdot| \ctwot \soops{b}{\beta}\Udot|\cdot| \eta \ctwot \grad \stwiddleoops{c}{\varsigma}\Udot| \notag
\\
\lesssim & \langle t \rangle ^{-2} \etalinfnorm{\ctwot \soops{b}{\beta}\Udot} \norm{\stwiddleoops{a}{\alpha}\Udot} \cdot \norm{\eta \ctwot \grad \stwiddleoops{c}{\varsigma}\Udot} \notag
\\
\lesssim & \langle t \rangle ^{-2} \Big(\sum_{|\vartheta |\leq 2} \norm{\soops{b}{\beta+\vartheta} \Udot } + \Psi_{b+|\beta|+3,b}(\Udot) \Big) \E{a+|\alpha|}{a}^{1/2}[U] \Psi_{c+|\varsigma|+1,c}(\Udot) \notag
\\
\lesssim & \langle t \rangle ^{-2}   \E{b+|\beta|+4}{b+1}^{1/2}[U]  \E{\mu}{\theta}^{1/2} [U] \E{c+|\varsigma|+2}{c+1}^{1/2}[U] \notag
\\
\lesssim & \langle t \rangle ^{-2} \E{\mu}{\theta+1}^{1/2} [U] \E{\mu}{\theta}^{1/2} [U] \E{\kappa}{\theta+1}^{1/2} [U]. \notag
\end{align}
Here we used (\ref{sob6}), (\ref{bootlem2}), and the fact that either $b+|\beta| \leq [\frac{\kappa}{2}]$ or $c+|\varsigma| \leq [\frac{\kappa}{2}]$.

The exterior estimates begin with the computed energy from (\ref{compen19}):

\begin{align}
&\sum_{{a+|\alpha| \leq \sigma}\atop {a\leq \theta}}
\Big( \int \gamma [D^k_n \ahat{lm}{pj}](I) \big \{ \frac{1}{2} \deen v^k\stwiddleoops{a}{\alpha} \hdot{j}{m}  \stwiddleoops{a}{\alpha}\hdot{p}{l} \label{enest13}
\\
&+\deel \hdot{k}{n} \stwiddleoops{a}{\alpha} \hdot{j}{m}  \stwiddleoops{a}{\alpha} v^i \notag
\\
&- \sum _{{b+c=a}\atop {\beta+\varsigma = \alpha}} \sum _{c+|\varsigma| \neq a+|\alpha|} C_{b,a,\beta,\alpha} \soops{b}{\beta} \hdot{k}{n} \stwiddleoops{a}{\alpha} v^p  \deel \stwiddleoops{c}{\varsigma}\hdot{j}{m} \big \} \notag
\\
&+\int \gamma \ahat{lm}{pj}(H) \stwiddleoops{a}{\alpha} \hdot{j}{m}  \deei \hdot{p}{l} \stwiddleoops{a}{\alpha} v^i
\notag
\\
&+\sum _{{b+c=a}\atop {\beta+\varsigma = \alpha}} \sum _{c+|\varsigma| \neq a+|\alpha|}  \int \gamma C_{b,a,\beta,\alpha} \Big ( \stwiddleoops{a}{\alpha} v^i \soops{b}{\beta} v^p \deep \stwiddleoops{c}{\varsigma} v^i \notag
\\
& -\ahat{lm}{pj}(I) \big \{ \stwiddleoops{a}{\alpha} \hdot{j}{m} \soops{b}{\beta} \hdot{p}{i} \deel \stwiddleoops{c}{\varsigma} v^i
+ \stwiddleoops{a}{\alpha} v^i \soops{b}{\beta} \hdot{p}{i} \deel \stwiddleoops{c}{\varsigma}\hdot{j}{m} \big \} \notag
\\
&+\ahat{lm}{pj}(H) \stwiddleoops{a}{\alpha}\hdot{j}{m} \soops{b}{\beta} v^i \deei \stwiddleoops{c}{\varsigma} \hdot{p}{l} \Big )  \notag
\\
&+\sum _{{b+c+d=a}\atop {\beta+\varsigma+\vartheta = \alpha}} \sum _{c+|\varsigma| \neq a+|\alpha|} \int  \gamma |\stwiddleoops{a}{\alpha}\Udot|\cdot| \soops{b}{\beta}\Udot|\cdot| \soops{d}{\vartheta} \Udot|\cdot| \grad \stwiddleoops{c}{\varsigma}\Udot| \Big) \notag.
\end{align}
We start by estimating the final term using $r \geq \frac{\ctwot}{m}$ on supp $\gamma$,(\ref{sob4}) and the fact that either $b+|\beta| \leq [\frac{\kappa}{2}]$ or $d+|\vartheta| \leq [\frac{\kappa}{2}]$
\begin{align}
\int  \gamma &|\stwiddleoops{a}{\alpha}\Udot|\cdot| \soops{b}{\beta}\Udot|\cdot| \soops{d}{\vartheta} \Udot|\cdot| \grad \stwiddleoops{c}{\varsigma}\Udot| \label{enest14}
\\
\lesssim & \int \langle t \rangle ^{-2} r^2 |\stwiddleoops{a}{\alpha}\Udot|\cdot| \soops{b}{\beta}\Udot|\cdot| \soops{d}{\vartheta} \Udot|\cdot| \grad \stwiddleoops{c}{\varsigma}\Udot| \notag
\\
\lesssim & \langle t \rangle ^{-2} \linfnorm{ \langle r \rangle \soops{b}{\beta}\Udot } \linfnorm{\langle r \rangle \soops{d}{\vartheta} \Udot} \norm{\stwiddleoops{a}{\alpha}\Udot} \norm{ \grad \stwiddleoops{c}{\varsigma}\Udot} \notag
\\
\lesssim & \langle t \rangle ^{-2} \E{b+|\beta|+2}{b}^{1/2}[U] \E{d+|\vartheta|+2}{d}^{1/2}[U] \E{a+|\alpha|}{a}^{1/2}[U] \E{c+|\varsigma|+1}{c}^{1/2}[U] \notag
\\
\lesssim & \langle t \rangle ^{-2} \E{b+|\beta|+2}{b}^{1/2}[U] \E{d+|\vartheta|+2}{d}^{1/2}[U] \E{\mu}{\theta}[U] \notag
\\
\lesssim & \langle t \rangle ^{-2}\E{\mu}{\mu}^{1/2}[U] \E{\kappa}{\theta}^{1/2}[U] \E{\mu}{\theta}[U]  \notag
\notag
\\
\lesssim & \langle t \rangle ^{-2} \E{\kappa}{\theta}^{1/2}[U] \E{\mu}{\theta}[U].  \notag
\end{align}
The convective terms from (\ref{enest13}) are estimated using (\ref{sob8}) in tandem with (\ref{polargrad}).  For example
\begin{align}
\int \gamma \ahat{lm}{pj}(H)& \stwiddleoops{a}{\alpha}\hdot{j}{m} \soops{b}{\beta} v^i \deei \stwiddleoops{c}{\varsigma} \hdot{p}{l} \label{enest15}
\\
=& \int \gamma \ahat{lm}{pj}(H) \stwiddleoops{a}{\alpha}\hdot{j}{m} \soops{b}{\beta} v^i \omega^i \deer \stwiddleoops{c}{\varsigma} \hdot{p}{l} \notag
\\
&+ \int \gamma r^{-1}\ahat{lm}{pj}(H) \stwiddleoops{a}{\alpha}\hdot{j}{m} \soops{b}{\beta} v^i (\omega \wedge \Omega)_i \stwiddleoops{c}{\varsigma} \hdot{p}{l} \notag
\\
\lesssim & \int \gamma |\stwiddleoops{a}{\alpha}\Udot|\cdot |\omega \cdot  (\soops{b}{\beta} \vdot)| \cdot |\grad \stwiddleoops{c}{\varsigma} \Udot| \notag
\\
& +  \langle t \rangle ^{-2} \int \gamma |\langle r \rangle \stwiddleoops{a}{\alpha}\Udot|\cdot |\soops{b}{\beta}\Udot|\cdot|\omegatwiddle \stwiddleoops{c}{\varsigma} \Udot| \notag
\\
\lesssim & \langle t \rangle ^{-3/2} \linfnorm{r^{3/2} \omega \cdot  (\soops{b}{\beta} \vdot)} \norm{\stwiddleoops{a}{\alpha}\Udot} \cdot \norm{ \grad \stwiddleoops{c}{\varsigma} \Udot} \notag
\\
&+ \langle t \rangle ^{-2} \linfnorm{\langle r \rangle \stwiddleoops{a}{\alpha}\Udot} \norm{\soops{b}{\beta}\Udot} \cdot \norm{\omegatwiddle \stwiddleoops{c}{\varsigma} \Udot} \notag
\\
\lesssim & \langle t \rangle ^{-3/2} \E{\mu}{\theta}[U] (\E{a+|\alpha|+2}{a}^{1/2}[U] + \E{b+|\beta|+2}{b}^{1/2}[U]) \notag
\\
\lesssim & \langle t \rangle ^{-3/2} \E{\mu}{\theta}[U] \E{\kappa}{\theta} ^{1/2} [U]. \notag
\end{align}
The other convective terms
$$\int \gamma \stwiddleoops{a}{\alpha} v^i \soops{b}{\beta} v^p \deep \stwiddleoops{c}{\varsigma} v^i $$
and
$$\int \gamma \ahat{lm}{pj}(H) \stwiddleoops{a}{\alpha} \hdot{j}{m}  \deei \hdot{p}{l} \stwiddleoops{a}{\alpha} v^i$$
have the same bound using a similar argument.

We are left with the cubic terms
\begin{align}
\int \gamma & [D^k_n \ahat{lm}{pj}](I) \big \{ \frac{1}{2} \deen v^k\stwiddleoops{a}{\alpha} \hdot{j}{m}  \stwiddleoops{a}{\alpha}\hdot{p}{l} \label{enest16}
+\deel \hdot{k}{n} \stwiddleoops{a}{\alpha} \hdot{j}{m}  \stwiddleoops{a}{\alpha} v^i
\\
&- \sum _{{b+c=a}\atop {\beta+\varsigma = \alpha}} \sum _{c+|\varsigma| \neq a+|\alpha|} C_{b,a,\beta,\alpha} \soops{b}{\beta} \hdot{k}{n} \stwiddleoops{a}{\alpha} v^p  \deel \stwiddleoops{c}{\varsigma}\hdot{j}{m} \big \} \notag
\\
&-\sum _{{b+c=a}\atop {\beta+\varsigma = \alpha}} \sum _{c+|\varsigma| \neq a+|\alpha|}  \int \gamma C_{b,a,\beta,\alpha}  \ahat{lm}{pj}(I) \big \{ \stwiddleoops{a}{\alpha} \hdot{j}{m} \soops{b}{\beta} \hdot{p}{i} \deel \stwiddleoops{c}{\varsigma} v^i \notag
\\
&+ \stwiddleoops{a}{\alpha} v^i \soops{b}{\beta} \hdot{p}{i} \deel \stwiddleoops{c}{\varsigma}\hdot{j}{m} \big \} \notag
\end{align}
which can only be estimated using the null condition.
Each individual term above can be written in the form
$$\int \gamma \langle \langle Q^l(\deel \stwiddleoops{c}{\varsigma} \Udot, \stwiddleoops{b}{\beta} \Udot), \stwiddleoops{a}{\alpha} \Udot \rangle \rangle $$
for some quadratic form $Q$ where each of $a+|\alpha|$, $b+|\beta|$, and $c+|\varsigma|+1$ are no larger than $\mu$ and $\langle \langle \cdot , \cdot \rangle \rangle \equiv  \langle \cdot , \cdot  \rangle _{\rthree \otimes \rthree \times \rthree }$.
If we decompose using the spectral projections (\ref{pplus})-(\ref{pzero}) and (\ref{polargrad}), we have
\begin{align}
\int \gamma & \langle \langle Q^l(\deel \stwiddleoops{c}{\varsigma} \Udot, \stwiddleoops{b}{\beta} \Udot), \stwiddleoops{a}{\alpha} \Udot \rangle \rangle \label{enest17}
\\
=& \sum_{K,L,M \in \{+,-,0 \} } \int \gamma  \langle \langle Q^l(\sproj{K} \deel \stwiddleoops{c}{\varsigma} \Udot, \sproj{L} \stwiddleoops{b}{\beta} \Udot), \sproj{M} \stwiddleoops{a}{\alpha} \Udot \rangle \rangle \notag
\\
=& \int \gamma  \langle \langle Q^l(\sproj{+} \deel \stwiddleoops{c}{\varsigma} \Udot, \sproj{+} \stwiddleoops{b}{\beta} \Udot), \sproj{+} \stwiddleoops{a}{\alpha} \Udot \rangle \rangle \notag
\\
&+ \sum_{(K,L,M) \neq (+,+,+) } \int \gamma  \langle \langle Q^l(\sproj{K} \deel \stwiddleoops{c}{\varsigma} \Udot, \sproj{L} \stwiddleoops{b}{\beta} \Udot), \sproj{M} \stwiddleoops{a}{\alpha} \Udot \rangle \rangle \notag
\\
=& \int \gamma  \langle \langle Q^l(\sproj{+} \omega_l \deer \stwiddleoops{c}{\varsigma} \Udot, \sproj{+} \stwiddleoops{b}{\beta} \Udot), \sproj{+} \stwiddleoops{a}{\alpha} \Udot \rangle \rangle \notag
\\
&-\int \gamma  \langle \langle Q^l(\sproj{+} r^{-1}(\omega \wedge \Omega)_l \stwiddleoops{c}{\varsigma} \Udot, \sproj{+} \stwiddleoops{b}{\beta} \Udot), \sproj{+} \stwiddleoops{a}{\alpha} \Udot \rangle \rangle \notag
\\
&+ \sum_{(K,L,M) \neq (+,+,+) } \int \gamma  \langle \langle Q^l(\sproj{K} \deel \stwiddleoops{c}{\varsigma} \Udot, \sproj{L} \stwiddleoops{b}{\beta} \Udot), \sproj{M} \stwiddleoops{a}{\alpha} \Udot \rangle \rangle \notag
\\
\equiv & \textrm{(N1)} + \textrm{(N2)} + \textrm{(N3)}
\end{align}
We bound (N2) by using $\langle t \rangle \lesssim \langle r \rangle$ on supp $\gamma$ and (\ref{sob4})
\begin{align}
\int \gamma&   r^{-1}|\omegatwiddle \stwiddleoops{c}{\varsigma} \Udot| \cdot | \stwiddleoops{b}{\beta} \Udot|\cdot  | \stwiddleoops{a}{\alpha} \Udot| \label{enest18a}
\\
\lesssim & \int \gamma \langle r \rangle \langle t \rangle ^{-2} |\omegatwiddle \stwiddleoops{c}{\varsigma} \Udot| \cdot | \stwiddleoops{b}{\beta} \Udot|\cdot  | \stwiddleoops{a}{\alpha} \Udot| \notag
\\
\lesssim & \langle t \rangle ^{-2} \linfnorm{\langle r \rangle \omegatwiddle \stwiddleoops{c}{\varsigma} \Udot} \norm{\stwiddleoops{b}{\beta} \Udot} \cdot \norm{\stwiddleoops{a}{\alpha} \Udot} \notag
\\
\lesssim & \langle t \rangle ^{-2} \E{\kappa}{\theta}[U]  \E{\mu}{\theta} [U] \notag
\end{align}
since $c+|\varsigma|+3 \leq \kappa$.

To estimate (N3), assume WLOG $K\neq +$.  Then because $$1\lesssim  r\langle t \rangle ^{-1} \lesssim \langle \ev{K} t -r \rangle \langle t \rangle ^{-1}$$ on supp $\gamma$, we have
\begin{align}
\int \gamma  \langle \langle Q^l &(\sproj{K} \deel \stwiddleoops{c}{\varsigma} \Udot, \sproj{L} \stwiddleoops{b}{\beta} \Udot), \sproj{M} \stwiddleoops{a}{\alpha} \Udot \rangle \rangle \label{enest18}
\\
\lesssim & \int \gamma \langle t \rangle ^{-3/2} \langle r \rangle  \langle \ev{K} t -r \rangle^{1/2} |\sproj{K} \grad \stwiddleoops{c}{\varsigma} \Udot| \cdot |\stwiddleoops{b}{\beta} \Udot| \cdot | \stwiddleoops{a}{\alpha} \Udot | \notag
\\
\lesssim & \langle t \rangle ^{-3/2} \linfnorm{\langle r \rangle  \langle \ev{K} t -r \rangle^{1/2} \sproj{K} \grad \stwiddleoops{c}{\varsigma} \Udot} \norm{\stwiddleoops{b}{\beta} \Udot} \cdot \norm{\stwiddleoops{a}{\alpha} \Udot} \notag
\\
\lesssim & \langle t \rangle ^{-3/2} (\E{c+|\varsigma| +3}{c}^{1/2}[U] + \Chi{c+|\varsigma|+ 4}{c}) \E{\mu}{\theta} [U] \notag
\\
\lesssim & \langle t \rangle ^{-3/2} \E{\kappa}{\theta+1}^{1/2}[U]  \E{\mu}{\theta} [U] \notag
\end{align}
after using (\ref{sob5}),(\ref{bootlem3}), and $c+|\varsigma|+5 \leq \kappa$.

The remaining term, (N1), is
$$\int \gamma  \langle \langle Q^l(\sproj{+} \omega_l \deer \stwiddleoops{c}{\varsigma} \Udot, \sproj{+} \stwiddleoops{b}{\beta} \Udot), \sproj{+} \stwiddleoops{a}{\alpha} \Udot \rangle \rangle.$$
Since each term in (\ref{enest16}) has a different $Q$, we will carefully estimate one of these terms and leave out the details for the remaining four terms.  The sum from the second line of (\ref{enest16}) has a quadratic form which gives the formula
\begin{align}
\int \gamma  \langle \langle& Q^l(\sproj{+} \omega_l \deer  \stwiddleoops{c}{\varsigma} \Udot, \sproj{+} \stwiddleoops{b}{\beta} \Udot), \sproj{+} \stwiddleoops{a}{\alpha} \Udot \rangle \rangle \label{enest19}
\\
=& \int \gamma [D^k_n \ahat{lm}{pj}](I) \omega_l ( \soops{b}{\beta} \vdot +  \soops{b}{\beta} \Hdot \omega)^k \omega_n \notag
\\
& \times  (\stwiddleoops{a}{\alpha} \vdot + \stwiddleoops{a}{\alpha} \Hdot \omega) ^p   (\deer \stwiddleoops{c}{\varsigma}\vdot + \deer \stwiddleoops{c}{\varsigma}\Hdot \omega)^j \omega_m \notag
\\
=& \int \gamma [D^k_n \ahat{lm}{pj}](I) \omega_l \omega_m \omega_n [(\Pee{1}+\Pee{2})( \soops{b}{\beta} \vdot +  \soops{b}{\beta} \Hdot \omega)]^k \notag
\\
&\times  [(\Pee{1}+\Pee{2})(\stwiddleoops{a}{\alpha} \vdot + \stwiddleoops{a}{\alpha} \Hdot \omega)] ^p \notag
\\
&\times [(\Pee{1}+\Pee{2})(\deer \stwiddleoops{c}{\varsigma}\vdot + \deer \stwiddleoops{c}{\varsigma}\Hdot \omega)]^j \notag
\end{align}
after decomposing using the formula (\ref{pplus}) and using our projections $\Pee{1}$ and $\Pee{2}$.  This product yields eight terms, one of which is
\begin{align}
\int \gamma & [D^k_n \ahat{lm}{pj}](I) \omega_l [\Pee{1}( \soops{b}{\beta} \vdot +  \soops{b}{\beta} \Hdot \omega)]^k \notag
\\
&\times  [\Pee{1}(\stwiddleoops{a}{\alpha} \vdot + \stwiddleoops{a}{\alpha} \Hdot \omega)] ^p  \omega_m
\notag
\\
&\times [\Pee{1}(\deer \stwiddleoops{c}{\varsigma}\vdot + \deer \stwiddleoops{c}{\varsigma}\Hdot \omega)]^j \notag
\\
=& ~0 \notag
\end{align}
by condition (\ref{ncproj}).  The remaining seven terms all have at least one factor of $\Pee{2}$, so each is bounded by
$$\int \gamma |\omega \cdot (\stwiddleoops{a}{\alpha} \vdot + \stwiddleoops{a}{\alpha} \Hdot \omega)|\cdot|\stwiddleoops{b}{\beta} \Udot|\cdot |\stwiddleoops{c}{\varsigma}\Udot|$$
where $a+|\alpha|, b+|\beta|,c+|\varsigma|\leq \mu$ are generic exponents.  Here we used the fact that the number of derivatives on each term in (\ref{enest19}) is less than $\mu$ and $|\deer f| \leq |\grad f| \leq |\Upsilon f|$ for general $f$.  Using (\ref{sob7}) and (\ref{sob8}), we have the bound
\begin{align}
\int \gamma& |\omega \cdot (\stwiddleoops{a}{\alpha} \vdot + \stwiddleoops{a}{\alpha} \Hdot \omega)|\cdot|\stwiddleoops{b}{\beta} \Udot|\cdot |\stwiddleoops{c}{\varsigma}\Udot| \label{enest20}
\\
\lesssim & \langle t \rangle ^{-3/2} (\linfnorm{r^{3/2} \omega \cdot (\stwiddleoops{a}{\alpha} \vdot)} + \linfnorm{r^{3/2} \omega \cdot (\stwiddleoops{a}{\alpha} \Hdot \omega)}) \E{\mu}{\theta} [U] \notag
\\
\lesssim & \langle t \rangle ^{-3/2} (\E{a+|\alpha|+2}{a}^{1/2}[U] + \sum_{|\beta|\leq 1} \norm{r \stwiddleoops{a}{\alpha+\beta} M^H(\Hdot)})\E{\mu}{\theta} [U] \notag
\end{align}
Pausing momentarily, we recall the definition of $M^H$as in (\ref{hdoteqc2b}) and use Lemma \ref{ahatlemma}, (\ref{sob4}), and the smallness of $\E{\mu}{\mu}^{1/2}[U]$ to derive
\begin{align}
\sum_{|\beta|\leq 1} \norm{r \stwiddleoops{a}{\alpha+\beta} M^H(\Hdot)} \label{enest21}
\lesssim & \sum_{{|\alpha_1| + |\alpha_2|  \leq |\alpha|+2}\atop {a_1+a_2 \leq a}} \norm{\langle r \rangle |\stwiddleoops{a_1}{\alpha_1} \Udot| \cdot |\stwiddleoops{a_2}{\alpha_2} \Udot|}
\\
\lesssim & \sum_{{|\alpha_1| + |\alpha_2|  \leq |\alpha|+2}\atop {a_1+a_2 \leq a}} \linfnorm{\langle r \rangle \stwiddleoops{a_1}{\alpha_1} \Udot} \norm{\stwiddleoops{a_2}{\alpha_2} \Udot} \notag
\\
\lesssim & \E{\mu}{\theta}^{1/2} [U] \E{\kappa}{\theta}^{1/2} [U] \notag
\\
\lesssim & \E{\kappa}{\theta}^{1/2} [U]. \notag
\end{align}
Therefore,
\begin{align}
\int \gamma &|\omega \cdot (\stwiddleoops{a}{\alpha} \vdot + \stwiddleoops{a}{\alpha} \Hdot \omega)|\cdot|\stwiddleoops{b}{\beta} \Udot|\cdot |\stwiddleoops{c}{\varsigma}\Udot| \label{enest22}
\\
\lesssim &
\langle t \rangle ^{-3/2} (\E{a+|\alpha|+2}{a}^{1/2}[U] + \E{\kappa}{\theta}^{1/2} [U]) \E{\mu}{\theta} [U] \notag
\\
\lesssim &
\langle t \rangle ^{-3/2} \E{\kappa}{\theta}^{1/2} [U] \E{\mu}{\theta} [U]. \notag
\end{align}
Gathering estimates (\ref{enest14}),(\ref{enest15}),(\ref{enest18a}),(\ref{enest18}), and (\ref{enest22}), we have shown that each term in (\ref{enest13}) is bounded by
$$\langle t \rangle ^{-3/2} \E{\kappa}{\theta}^{1/2} [U] \E{\mu}{\theta} [U].$$
When we combine this result with (\ref{enest12}), we have

\begin{equation}
\deet \E{\mu}{\theta} [U] - \visc \sum _{{a+|\alpha| \leq \mu}\atop {a\leq \theta}} \int \stwiddleoops{a}{\alpha} v^i  \soops{a}{\alpha} \laplacian v^i \lesssim \langle t \rangle ^{-3/2} \E{\kappa}{\theta+1}^{1/2} [U] \E{\mu}{\theta+1} [U]. \label{enest23}
\end{equation}

\noindent \textit{Induction on $\theta$}

We will complete the proof by proving estimates of the form

\begin{equation}
\deet \E{\mu}{\mu} [U] +\frac{\visc}{2} \sum _{{a+|\alpha| \leq \mu}}  \norm{ \grad \stwiddleoops{a}{\alpha} v} ^2 \lesssim \langle t \rangle ^{-3/2} \E{\mu}{\mu} [U] \E{\kappa}{\mu+1}^{1/2}[U] \label{ge1}
\end{equation}
and
\begin{equation}
\deet \E{\kappa}{\mu+1} [U] +\frac{\visc}{2} \sum _{{a+|\alpha| \leq \kappa}\atop {a\leq \mu+1}} \norm{ \grad \stwiddleoops{a}{\alpha} v} ^2 \lesssim \langle t \rangle ^{-1} \E{\mu}{\mu}^{1/2} [U] \E{\kappa}{\mu+1}[U]. \label{ge2}
\end{equation}
via finite induction on $\theta$ using (\ref{enest11}),(\ref{enest11a}), and (\ref{enest23}).

Applying (\ref{enest23}) with $\theta = 1,2,...,\mu$ and we have
\begin{align}
\deet \E{\mu}{0} [U] - \visc \sum _{{|\alpha| \leq \mu}} \innerprod{ \Upsilon^{\alpha} v}  {\Upsilon^{\alpha} \laplacian v} &\lesssim \langle t \rangle ^{-3/2} \E{\kappa}{1}^{1/2} [U] \E{\mu}{1} [U], \label{enest24}
\\
\deet \E{\mu}{1} [U] - \visc \sum _{{a+|\alpha| \leq \mu}\atop {a\leq 1}} \innerprod{ \stwiddleoops{a}{\alpha} v} { \soops{a}{\alpha} \laplacian v} &\lesssim \langle t \rangle ^{-3/2} \E{\kappa}{2}^{1/2} [U] \E{\mu}{2} [U], \notag
\\
&. \notag
\\
&. \notag
\\
&. \notag
\\
\deet \E{\mu}{\mu} [U] - \visc \sum _{{a+|\alpha| \leq \mu}} \innerprod{ \stwiddleoops{a}{\alpha} v}  {\soops{a}{\alpha} \laplacian v} &\lesssim \langle t \rangle ^{-3/2} \E{\kappa}{\mu+1}^{1/2} [U] \E{\mu}{\mu} [U]. \notag
\end{align}
The first line may be integrated by parts giving
\begin{equation}
\deet \E{\mu}{0} [U] + \visc \sum _{{|\alpha| \leq \mu}} \norm{ \grad \Upsilon^{\alpha} v }^2 \lesssim \langle t \rangle ^{-3/2} \E{\kappa}{1}^{1/2} [U] \E{\mu}{1} [U]. \label{enest24a}
\end{equation}
In the remaining lines we calculate by brute force using the commutation formula (\ref{snlaplcommute}) and integration by parts
\begin{align}
\visc \sum _{{a+|\alpha| \leq \mu}\atop {a\leq \theta}} \innerprod{ \stwiddleoops{a}{\alpha} v}   {&\soops{a}{\alpha} \laplacian v} \label{enest25}
\\
=& \visc \sum _{{a+|\alpha| \leq \mu}\atop {a\leq \theta}} \sum_{j=0}^{a} {a\choose j} (-1)^{a-j} \int  \stwiddleoops{a}{\alpha} v^i   \laplacian \stwiddleoops{j}{\alpha} v^i \notag
\\
= & \visc \sum _{{a+|\alpha| \leq \mu}\atop {a\leq \theta}} \int \stwiddleoops{a}{\alpha} v^i   \laplacian \stwiddleoops{a}{\alpha} v^i \notag
\\
& + \visc \sum _{{a+|\alpha| \leq \mu}\atop {a\leq \theta}} \sum_{j=0}^{a-1} {a\choose j} (-1)^{a-j} \int  \stwiddleoops{a}{\alpha} v^i   \laplacian \stwiddleoops{j}{\alpha} v^i \notag
\\
= & -\visc \sum _{{a+|\alpha| \leq \mu}\atop {a\leq \theta}} \norm{ \grad \stwiddleoops{a}{\alpha} v} ^2  \notag
\\
& - \visc \sum _{{a+|\alpha| \leq \mu}\atop {a\leq \theta}} \sum_{j=0}^{a-1} {a\choose j} (-1)^{a-j} \int   \deek \stwiddleoops{a}{\alpha} v^i   \deek \stwiddleoops{j}{\alpha} v^i \notag
\\
\leq & -\visc \sum _{{a+|\alpha| \leq \mu}\atop {a\leq \theta}} \norm{ \grad \stwiddleoops{a}{\alpha} v} ^2  +\frac{\visc}{2} \sum _{{a+|\alpha| \leq \mu}\atop {a\leq \theta}} \norm{ \grad \stwiddleoops{a}{\alpha} v} ^2 \notag
\\
&+ \frac{\visc}{2} \sum _{{a+|\alpha| \leq \mu}\atop {a\leq \theta}} \sum_{j=0}^{a-1} {a\choose j} \norm{ \grad \stwiddleoops{j}{\alpha} v} ^2 \notag
\\
\leq & -\frac{\visc}{2} \sum _{{a+|\alpha| \leq \mu}\atop {a\leq \theta}} \norm{ \grad \stwiddleoops{a}{\alpha} v} ^2 \notag + \frac{\visc \theta C_\theta}{2} \sum _{{a+|\alpha| \leq \mu}\atop {a\leq \theta-1}} \norm{ \grad \stwiddleoops{a}{\alpha} v} ^2 \notag
\end{align}
where
\begin{equation} C_\theta = \max_j {a\choose j}.\end{equation}
Now we can write
\begin{align}
\deet \E{\mu}{\theta} [U] & + \frac{\visc}{2} \sum _{{a+|\alpha| \leq \mu}\atop {a\leq \theta}} \norm{ \grad \stwiddleoops{a}{\alpha} v} ^2-\frac{\visc \theta C_\theta}{2} \sum _{{a+|\alpha| \leq \mu}\atop {a\leq \theta-1}} \norm{ \grad \stwiddleoops{a}{\alpha} v} ^2  \label{enest26}
\\
&\leq C \langle t \rangle ^{-3/2} \E{\kappa}{\theta+1}^{1/2} [U] \E{\mu}{\theta+1} [U]  . \notag
\end{align}
By a similar argument, for $1\leq \theta \leq \mu$ we use (\ref{enest11}) to get
\begin{align}
\deet \E{\kappa}{\theta} [U] & + \frac{\visc}{2} \sum _{{a+|\alpha| \leq \kappa}\atop {a\leq \theta}} \norm{ \grad \stwiddleoops{a}{\alpha} v} ^2 - \frac{\visc \theta C_\theta}{2} \sum _{{a+|\alpha| \leq \kappa}\atop {a\leq \theta-1}} \norm{ \grad \stwiddleoops{a}{\alpha} v} ^2\label{enest27}
\\
&\leq C \langle t \rangle ^{-1} \E{\mu}{\mu}^{1/2} [U] \E{\kappa}{\theta+1} [U] , \notag
\end{align}
and (\ref{enest11a}) to get
\begin{align}
\deet \E{\kappa}{\mu+1} [U] & + \frac{\visc}{2} \sum _{{a+|\alpha| \leq \kappa}\atop {a\leq \mu+1}} \norm{ \grad \stwiddleoops{a}{\alpha} v} ^2 - \frac{\visc \mu C_\mu}{2} \sum _{{a+|\alpha| \leq \kappa}\atop {a\leq \mu}} \norm{ \grad \stwiddleoops{a}{\alpha} v} ^2\label{enest28}
\\
&\leq C \langle t \rangle ^{-1} \E{\mu}{\mu}^{1/2} [U] \E{\kappa}{\mu+1} [U] . \notag
\end{align}
Because the $\deet E$ terms on the left of (\ref{enest24a}),(\ref{enest26})-(\ref{enest28}) are not necessarily positive, we write each these estimates in its integrated form:
\begin{equation}
\E{\mu}{0} [U] + \visc \sum _{{|\alpha| \leq \mu}} \int_0^t \norm{ \grad \Upsilon^{\alpha} v }^2 \lesssim \int_0^t \langle \tau \rangle ^{-3/2} \E{\kappa}{1}^{1/2} [U] \E{\mu}{1} [U] d\tau , \label{enest24int}
\end{equation}
\begin{align}
\E{\mu}{\theta} [U] & + \frac{\visc}{2} \sum _{{a+|\alpha| \leq \mu}\atop {a\leq \theta}} \int_0^t\norm{ \grad \stwiddleoops{a}{\alpha} v} ^2-\frac{\visc \theta C_\theta}{2} \sum _{{a+|\alpha| \leq \mu}\atop {a\leq \theta-1}} \int_0^t \norm{ \grad \stwiddleoops{a}{\alpha} v} ^2  \label{enest26int}
\\
&\leq C \int_0^t \langle \tau \rangle ^{-3/2} \E{\kappa}{\theta+1}^{1/2} [U] \E{\mu}{\theta+1} [U] d\tau ,\notag
\end{align}
\begin{align}
\E{\kappa}{\theta} [U] & + \frac{\visc}{2} \sum _{{a+|\alpha| \leq \kappa}\atop {a\leq \theta}} \int_0^t \norm{ \grad \stwiddleoops{a}{\alpha} v} ^2 - \frac{\visc \theta C_\theta}{2} \sum _{{a+|\alpha| \leq \kappa}\atop {a\leq \theta-1}} \int_0^t \norm{ \grad \stwiddleoops{a}{\alpha} v} ^2\label{enest27int}
\\
&\leq C \int_0^t \langle \tau \rangle ^{-1} \E{\mu}{\mu}^{1/2} [U] \E{\kappa}{\theta+1} [U] d\tau, \notag
\end{align}
and
\begin{align}
\E{\kappa}{\mu+1} [U] & + \frac{\visc}{2} \sum _{{a+|\alpha| \leq \kappa}\atop {a\leq \mu+1}} \int_0^t \norm{ \grad \stwiddleoops{a}{\alpha} v} ^2 - \frac{\visc \mu C_\mu}{2} \sum _{{a+|\alpha| \leq \kappa}\atop {a\leq \mu}} \int_0^t \norm{ \grad \stwiddleoops{a}{\alpha} v} ^2\label{enest28int}
\\
&\leq C \int_0^t \langle \tau \rangle ^{-1} \E{\mu}{\mu}^{1/2} [U] \E{\kappa}{\mu+1} [U] d\tau. \notag
\end{align}
If we add (\ref{enest26int}) with $\theta =1$ to (\ref{enest24int}), we have
\begin{align}
\E{\mu}{1} [U] +& \E{\mu}{0} [U]+ \frac{\visc}{2} \sum _{{a+|\alpha| \leq \mu}\atop {a\leq 1}} \int_0^t \norm{ \grad \stwiddleoops{a}{\alpha} v} ^2 + \frac{\visc}{2} \sum _{{|\alpha| \leq \mu}} \int_0^t \norm{ \grad \Upsilon^{\alpha} v }^2 \label{enest29}
\\
&\leq C \int_0^t \langle \tau \rangle ^{-3/2} (\E{\kappa}{2}^{1/2} [U] \E{\mu}{2} [U]+ \E{\kappa}{1}^{1/2} [U] \E{\mu}{1} [U]) d\tau  \notag
\end{align}
which implies
\begin{align}
\E{\mu}{1} [U] +\E{\mu}{0} [U]+ \frac{\visc}{2} \sum _{{a+|\alpha| \leq \mu}\atop {a\leq 1}} \int_0^t \norm{ \grad \stwiddleoops{a}{\alpha} v} ^2  \label{enest30}
\lesssim \int_0^t \langle \tau \rangle ^{-3/2} \E{\kappa}{2}^{1/2} [U] \E{\mu}{2} [U] d\tau.
\end{align}
For $\theta = 2,3,...,\mu$ if we assume
\begin{align}
\E{\mu}{\theta-1} [U] + \frac{\visc}{2} \sum _{{a+|\alpha| \leq \mu}\atop {a\leq \theta-1}} \int_0^t \norm{ \grad \stwiddleoops{a}{\alpha} v} ^2  \label{enest31}
\lesssim \int_0^t \langle \tau \rangle ^{-3/2} \E{\kappa}{\theta}^{1/2} [U] \E{\mu}{\theta} [U] d\tau
\end{align}
then adding (\ref{enest26int}) multiplied by $\theta C_\theta$ to (\ref{enest31}) we have
\begin{align}
\E{\mu}{\theta} [U] + & \theta C_\theta \E{\mu}{\theta-1} [U]+ \frac{\visc}{2} \sum _{{a+|\alpha| \leq \mu}\atop {a\leq \theta}} \int_0^t \norm{ \grad \stwiddleoops{a}{\alpha} v} ^2 \label{enest32}
\\
&\leq C \int_0^t \langle \tau \rangle ^{-3/2} (\E{\kappa}{\theta+1}^{1/2} [U] \E{\mu}{\theta+1} [U]+ \theta C_\theta \E{\kappa}{\theta}^{1/2} [U] \E{\mu}{\theta} [U]) d\tau,  \notag
\end{align}
which gives us
\begin{align}
\E{\mu}{\theta} [U] + \frac{\visc}{2} \sum _{{a+|\alpha| \leq \mu}\atop {a\leq \theta}} \int_0^t \norm{ \grad \stwiddleoops{a}{\alpha} v} ^2  \label{enest33}
\lesssim \int_0^t \langle \tau \rangle ^{-3/2} \E{\kappa}{\theta+1}^{1/2} [U] \E{\mu}{\theta+1} [U]d\tau.
\end{align}
Therefore, by our inductive argument at $\theta=\mu$, we have
\begin{align}
\E{\mu}{\mu} [U] + \frac{\visc}{2} \sum _{{a+|\alpha| \leq \mu}} \int_0^t \norm{ \grad \stwiddleoops{a}{\alpha} v} ^2  \label{enest33a}
\lesssim \int_0^t \langle \tau \rangle ^{-3/2} \E{\kappa}{\mu+1}^{1/2} [U] \E{\mu}{\mu} [U]d\tau .
\end{align}

A similar inductive argument for the high energy using (\ref{enest27int}) shows
\begin{align}
\E{\kappa}{\theta} [U] + \frac{\visc}{2} \sum _{{a+|\alpha| \leq \kappa}\atop {a\leq \theta}} \int_0^t \norm{ \grad \stwiddleoops{a}{\alpha} v} ^2  \label{enest34}
\lesssim \int_0^t \langle \tau \rangle ^{-1} \E{\kappa}{\theta+1} [U] \E{\mu}{\mu} ^{1/2} [U] d\tau
\end{align}
for $\theta=0,1,2,...,\mu$. At $\theta=\mu+1$ we use (\ref{enest28int}) to get
\begin{align}
\E{\kappa}{\mu+1} [U] + \frac{\visc}{2} \sum _{{a+|\alpha| \leq \kappa}\atop {a\leq \mu+1}} \int_0^t \norm{ \grad \stwiddleoops{a}{\alpha} v} ^2  \label{enest35}
\lesssim \int_0^t \langle \tau \rangle ^{-1} \E{\kappa}{\mu+1} [U] \E{\mu}{\mu} ^{1/2} [U]d\tau.
\end{align}
Estimates (\ref{enest33a}) and (\ref{enest35}) are equivalent to estimates (\ref{ge1}) and (\ref{ge2}) which in turn imply global existence. \end{proof}

\small
\bibliographystyle{amsplain}
 \bibliography{diss}

\end{document}